\documentclass[11pt]{amsart}
\usepackage{amsmath,amssymb,latexsym,dsfont}
\usepackage[left=2.6cm,right=2.6cm,top=2.7cm,bottom=2.7cm]{geometry}

    \usepackage{amsmath,amsfonts,amssymb,epsfig}
     \usepackage{graphicx}
     \usepackage{hyperref}
     \usepackage{cases}
     \usepackage{color}

\newtheorem{theorem}{Theorem}[section]
\newtheorem{lemma}{Lemma}[section]

\newtheorem{proposition}{Proposition}[section]

\newtheorem{remark}{Remark}[section]

\newtheorem{definition}{Definition}[section]
\numberwithin{equation}{section}

      \newcommand{\R}{{\mathbb{R}}}

      \newcommand{\curl}{\operatorname{curl}}
      \newcommand{\dive}{\operatorname{div}}
      \newcommand{\N}{\mathbb{N}}
      \newcommand{\ext}{\operatorname{ext}}
      \newcommand{\inte}{\operatorname{int}}
      \newcommand{\loc}{\operatorname{loc}}
      
      \newcommand{\eps}{\varepsilon}
      \newcommand{\mR}{\mathbb{R}}
      
      \newcommand{\pbE}{\operatorname{\mathcal E}}
      \newcommand{\pbH}{\operatorname{\mathcal H}}

        \newcommand{\VV}{\mathbb{V}}
      
      \newcommand{\bE}{\operatorname{\textbf{E}}}
      \newcommand{\bH}{\operatorname{\textbf{H}}}

     \newcommand{\supp}{\mbox{supp }}
           \newcommand{\dsp}{\displaystyle}

      \makeatletter
      \def\@setcopyright{}
      \def\serieslogo@{}
      \makeatother

\begin{document}

   \author[H.-M. Nguyen]{Hoai-Minh Nguyen}
\author[L. Tran]{Loc Tran}

\address[H.-M. Nguyen]{Department of Mathematics, EPFL SB CAMA, Station 8,  \newline\indent
	 CH-1015 Lausanne, Switzerland.}
\email{hoai-minh.nguyen@epfl.ch}

\address[L. Tran]{Department of Mathematics, 
	EPFL SB CAMA Station 8, 
	\newline\indent CH-1015 Lausanne, Switzerland.}
\email{loc.tran@epfl.ch}


\title[Cloaking for Maxwell's equations]{Approximate cloaking for electromagnetic waves via transformation optics: cloaking vs infinite energy}
   
\begin{abstract}
We study  the approximate cloaking  via transformation optics for electromagnetic waves in the time harmonic regime in which the cloaking device {\it only} consists of a layer constructed by the mapping technique. Due to the fact that no-lossy layer is required, resonance might appear and the analysis is delicate.  We analyse  both non-resonant and resonant cases. In particular, we show that the energy can blow up inside the cloaked region in the resonant case and/whereas  cloaking is {\it achieved} in {\it both} cases. Moreover, the degree of visibility {\it depends} on the compatibility of the source inside the cloaked region and the system. These facts are new and  distinct from known mathematical results in the literature.  
\end{abstract}

   \dedicatory{}

   \date{\today}

   \maketitle

\tableofcontents



\section{Introduction}

 
Cloaking via transformation optics was introduced by Pendry, Schurig, and Smith \cite{Pen} for the Maxwell system and by Leonhardt \cite{Leo} in the geometric optics setting. They used a singular change of variables which blows up a point into a
cloaked region. The same transformation was used  by Greenleaf, Lassas,
and Uhlmann  to establish the nonuniqueness of Calderon's problem in \cite{Green}.  The singular nature of the cloaks presents various difficulties in practice as well as in theory: (1) they are hard to fabricate and (2) in certain cases the correct definition of the corresponding electromagnetic fields is not obvious. To avoid using
the singular structure, various regularized schemes have been proposed. One of them  was suggested by Kohn, Shen, Vogelius, and Weinstein  in \cite{Kohn} in which they  used a  transformation which blows up a small ball of radius $\rho$ instead of a point 
into the cloaked region.

Approximate cloaking schemes for the Helmholtz equation based on the regularized transformations introduced in \cite{Kohn} have been studied extensively in  \cite{GKLU07-1, GKLU07, Kohn1, HM1, HM2, HM4, Cap, Ammari13-1, HV, GV}.  Frequently, a (damping) lossy layer is employed inside the transformation cloak. Without the lossy layer, the field inside the cloaked region might depend on the field outside, and resonance might appear and affect the cloaking ability of the cloak, see  \cite{HM2}. Approximate cloaking was investigated in the time domain for the acoustic waves in  \cite{HM5, HM6}. In \cite{HM6}, the dependence of the material constants on frequency via the Drude-Lorentz model was taken into account.   

Cloaking for electromagnetic waves via transformation optics has been mathematically investigated  by several authors.  Greenleaf, Kurylev, Lassas, and Uhlmann in  \cite{GKLU07-1}  and Weder in  \cite{Weder1, Weder2} studied cloaking for the singular scheme mentioned above by considering finite energy solutions. Concerning this approach, the information inside the cloaked region is not seen by observers outside.  
Approximate cloaking for the Maxwell equations using schemes in the spirit of \cite{Kohn} was considered in \cite{Bao, Ammari13, Lassas}. In \cite{Ammari13}, Ammari et al. investigated cloaking using additional  layers inside the transformation cloak. These additional layers depending on the cloaked object   were chosen in an appropriate  way to cancel first terms in the asymptotic expansion of the polarization tensor to  enhance the cloaking property.   In \cite{Bao}, Bao, Liu, and Zou  studied approximate cloaking using  a lossy layer inside the transformation cloak. Their approach is as follows. Taking into account the lossy layer, one  easily obtains an estimate for the electric field inside the lossy layer.  This estimate depends on the property of the lossy layer and degenerates as the lossy property disappears. They then used the equation of the electric field in  the lossy layer to derive estimates for the electric field on the boundary of the lossy region in some negative Sobolev norm.  The cloaking estimate can  be finally deduced from the integral representation for the electric field. This approach essentially uses the property of the lossy-layer and does not provide an optimal estimate of the degree of visibility in general.   
For example, when a fixed lossy layer is employed, they showed  that the degree of visibility is of the order $\rho^2$, which is not optimal. 
 In \cite{Lassas}, Lassas and Zhou considered the transformation cloak in a symmetric setting, dealt with the non-resonant case (see Definition~\ref{def-R}) and studied the limit of  the solutions of the approximate cloaking problem near the cloak interface using separation of variables. Other regularized schemes are considered in \cite{DLU}. 


In this paper, we  investigate approximate cloaking for the Maxwell equation in the time harmonic regime using a  scheme in the spirit of \cite{Kohn}. More precisely, we consider the situation where 
the cloaking device {\it only} consists of a layer constructed by the mapping technique and there is no source in that layer. Due to the fact that no-lossy (damping) layer is required, resonance might appear and the analysis is subtle.  
Our analysis is given in both non-resonant  and resonant cases (Definition~\ref{def-R}) and the results can be briefly summarized as follows. 

\begin{enumerate} 
\item[i)] In the  non-resonant case, cloaking is achieved,  and the energy remains finite inside the cloaked region. 

\item[ii)] In the resonant case,  cloaking is also {\it achieved}. Nevertheless,  the degree of invisibility varies and depends on the compatibility (see \eqref{passive} and \eqref{N-passive}) of the source  with the system. Moreover, the energy inside the cloaked region might explode in the incompatible case. See Theorems~\ref{thm1.1} and \ref{thm1.2}. 

\item[iii)] The degree of visibility is of the order $\rho^3$ for both non-resonant and resonant cases if no source is  inside the cloaked region (Theorems~\ref{thm1} and \ref{thm1.1}). 
\end{enumerate}

We also investigate the behavior of the field in the {\it whole} space (Theorems~\ref{thm1}, \ref{thm1.1}, and \ref{thm1.2}) and establish the optimality of the  convergence rate (Section~\ref{sect-opt}).  Our results are new and distinct from the works mentioned above. First, cloaking  takes place even if the energy explodes inside the cloaked region as $\delta$ goes to 0. Second, in the resonant case with finite energy inside the cloaked region, the fields inside the cloaked region satisfy a non-local structure. Optimal estimates for the degree of visibility are derived for all cases. In particular, in the case of a fixed lossy layer (non-resonant case), the degree of visibility is of the order $\rho^3$ instead of   $\rho^2$ as obtained previously. Both non-resonant and  resonant cases are analysed in details  without assuming the symmetry of the cloaking setting.

Our approach is    different from the ones in the  works mentioned. It is based on severals subtle estimates for the effect of small inclusion involving the blow-up structure. Part of the analysis is on Maxwell's equations in the low frequency regime, which is interesting in itself.  The approach in this paper is inspired from \cite{HM2} where the acoustic setting was considered. Nevertheless, the analysis for the electromagnetic setting is challenging and requires further new ideas 
due to the non-standard structure coming from the mapping technique and the complexity of electromagnetic structures/phenomena in comparison with acoustic ones. The Helmholtz decomposition and Stokes' theorem are involved in the Maxwell context. 


\section{Statement of the main results}

In this section, we describe the problem in more details and state the main results of this paper. For simplicity of notations, we suppose that the cloak occupies
the annular region $B_2 \setminus B_1$ and the cloaked region is the unit ball $B_1$ in $\mR^3$ in which the permittivity and the permeability are given by two $3\times 3$ matrices $\eps, \mu$ respectively.  Here and in what follows, for $r>0$, let $B_r$ denote the open ball in $\mR^3$ centered at the origin and of radius $r$. 
Through this paper, we assume that 
\begin{equation}\label{symetric}
\eps, \mu \mbox{ are real, symmetric},  
\end{equation}
and uniformly elliptic in $B_1$, i.e.,   for a.e.  $x\in B_1$ and for some $\Lambda  \ge 1$, 
\begin{equation}\label{elliptic}
\frac{1}{\Lambda}|\xi|^2 \le \langle \eps(x)\xi, \xi \rangle, \langle\mu(x)\xi, \xi\rangle \le \Lambda|\xi|^2\quad  \mbox{ for all } \, \xi \in \mR^3.
\end{equation}
We assume in addition that $\eps, \mu$ are piecewise $C^1$ in order to ensure the well-posedness of Maxwell's equations in the frequency domain (via the unique continuation principle). In the spirit of  the scheme in \cite{Kohn}, the permittivity and permeability of the cloaking region are given by
$$
(\eps_c, \mu_c) := ({F_{\rho}}_*I, {F_{\rho}}_*I ) \mbox{ in } B_{2} \setminus B_1, 
$$
where $F_{\rho}: \R^3 \rightarrow \R^3$ with $ \rho \in (0,  1/2)$ is defined by
\begin{equation*}F_{\rho} =\left\{ \begin{array}{cl} \dsp x &\text{ in } \R^3\setminus B_2, \\[6pt] 
\dsp \left( \frac{2 - 2\rho}{2-\rho} + \frac{|x|}{2-\rho} \right)\frac{x}{|x|} &\text{ in } B_2\setminus B_{\rho},\\[6pt]
\dsp \frac{x}{\rho} &\text{ in } B_{\rho}.
\end{array}\right. 
\end{equation*}
We denote 
\[
F_0(x) = \lim \limits_{\rho \rightarrow 0} F_{\rho}(x) \mbox{ for } x \in \mR^3. 
\]
As usual,  for a matrix $A \in \R^{3\times 3}$ and for a bi-Lipschitz homeomorphism $T$, the following notation is used:
$$
T_\ast A(y) = \frac{D T (x) A (x) D T^{T}(x)}{|\det D T(x)|} \mbox{ with } y = T(x).
$$
Assume that the medium is homogeneous outside the cloak and the cloaked region. In the presence of the cloaked object and the cloaking device, the medium in the whole space $\R^3$ is given by  $(\eps_c, \mu_c)$ which is defined as follows
\begin{equation}
\label{medium:cloak}
(\eps_c, \mu_c) = \left\{\begin{array}{cl} (I, I) & \mbox{ in } \mR^3 \setminus B_2, \\[6pt]
\big( {F_{\rho}}_*I, {F_{\rho}}_*I \big)  & \text{ in } B_2 \setminus  B_1,\\[6pt]
(\eps, \mu)  &\text{ in } B_{1}.
\end{array} \right. 
\end{equation}

With the cloak and the object, in the time harmonic regime of frequency $\omega>0$, the electromagnetic field generated by current $J \in [L^2(\mR^3)]^3$ is the unique (Silver-M\"uller) radiating solution $(E_c, H_c)\in [H_{\loc}(\curl, \R^3)]^2$ of the system
\begin{equation}
\label{equ:cloak}
\begin{cases}
\nabla \times E_c = i\omega \mu_c H_c  &\text{ in } \mathbb{R}^3,\\[6pt]
\nabla \times H_c = -i\omega \eps_c E_c + J &\text{ in } \mathbb{R}^3.
\end{cases}\end{equation}

For an open subset $U$ of $\mR^3$,  denote 
\[
H(\curl, U) := \Big\{\phi \in [L^2(U)]^3;  \;  \nabla \times \phi \in [L^2(U)]^3 \Big\}
\]
and 
\[
H_{\loc}(\curl, U) := \Big\{\phi \in [L_{\loc}^2(U)]^3; \;  \nabla \times \phi \in [L_{\loc}^2(U)]^3 \Big\}.
\]
Recall that, for $\omega> 0$, a solution $(E, H) \in [H_{\loc}(\curl, \R^3\setminus B_R)]^2$, for some $R> 0$, of the Maxwell equations 
\[
\begin{cases}
\nabla \times E = i \omega H  &\text{ in } \mathbb{R}^3\setminus B_R,\\[6pt]
\nabla \times H = -i \omega E  &\text{ in } \mathbb{R}^3\setminus B_R
\end{cases}
\]
is called radiating if it satisfies one of the (Silver-Muller) radiation conditions
\begin{equation}\label{SM-condition}
H \times x - |x| E = O(1/|x|) \quad   \mbox{ and } \quad  E\times x + |x| H = O(1/|x|) \mbox{ as } |x| \to + \infty. 
\end{equation}
Here and in what follows, for $\alpha \in \mR$, $O(|x|^\alpha)$ denotes a quantity whose norm is bounded by $C |x|^\alpha$ for some constant $C>0$.

Denote $J_{\ext}$ and $J_{\inte}$ the restriction of $J$ into $\mR^3 \setminus B_1$ and $B_1$ respectively. It is clear that 
\begin{equation}
J = \left\{\begin{array}{cl} J_{\ext} &\text{ in } \R^3\setminus B_1,\\[6pt]
J_{\inte} & \text{ in } B_1.
\end{array}\right.
\end{equation}
In the homogeneous medium (without the cloaking device and the cloaked object), the electromagnetic field  generated by $J_{\ext}$ is the unique (Silver-M\"uller) radiating solution $(E, H)\in [H_{\loc}(\curl, \R^3)]^2$ to the system 
\begin{align}
\label{equ:homo}
\begin{cases}
\nabla \times E = i\omega H  &\text{ in } \mathbb{R}^3,\\[6pt]
\nabla \times H = -i\omega E + J_{\ext}  &\text{ in } \mathbb{R}^3.
\end{cases}
\end{align}

We next introduce the functional space ${\mathcal N}$ which is related to the notion of resonance and  plays a  role in our analysis. 

\begin{definition} \label{def:1} Let $D$ be a smooth bounded subset of $\mR^3$ such that $\mR^3 \setminus D$ is connected.  Set 
 \[ {\mathcal N}(D): = \Big\{ (\bE, \bH) \in [H(\curl, D)]^2:  (\bE, \bH) \text{ satisfies the system } (\ref{def:resonance}) \Big\}, \]
where
\begin{equation}
\label{def:resonance}
\begin{cases} \nabla\times \bE = i\omega \mu \bH & \text{ in } D,\\[6pt]
 \nabla\times \bH = -i\omega \eps \bE &\text{ in } D, \\[6pt] 
 \nabla\times \bE \ \cdot \ \nu = \nabla\times \bH \ \cdot \ \nu = 0 &\text{ on } \partial D.
 \end{cases}
\end{equation} 
In the case $D  = B_1$, we simply denote ${\mathcal N}(B_1)$ by ${\mathcal N}$.   
\end{definition}


The notions of resonance and non-resonance  are  defined as follows: 

\begin{definition}  \label{def-R}The cloaking system \eqref{medium:cloak} is said to be non-resonant if ${\mathcal N} = \{(0, 0) \}$. Otherwise, the  cloaking system \eqref{medium:cloak} is called resonant. 

\end{definition}

%

%

Our main result in the non-resonance case is the following theorem. 
  
   \begin{theorem}
   \label{thm1}
   Let $\rho \in (0, 1/2)$,  $R_0> 2$, and  let $J \in L^2(\mR^3)$ be  such that $\operatorname{supp} J_{\ext} \subset \subset B_{R_0}\setminus B_2$.  Assume that system (\ref{medium:cloak}) is non-resonant. We have, for all $K\subset \subset \R^3\setminus \bar{B_1}$, 
 \begin{equation}\label{thm1-est1}
 \|(F_\rho^{-1}*E_c, F_\rho^{-1}*H_c) - (E, H)\|_{H(\curl, K)}\leq C\left(\rho^3 \|J_{\ext}\|_{L^2(B_{R_0}\setminus B_2)}+\rho^2\|J_{\inte}\|_{L^2(B_1)}\right),
\end{equation}
   for some positive constant $C$ depending only on $R_{0}, \omega, K, \mu, \eps$. Moreover, 
 \begin{equation}
 \label{thm1-CV}
\lim_{\rho \to 0} (E_c, H_c) =  Cl(0, J_{\inte}) \mbox{ in } [H(\curl, B_1)]^2,
\end{equation}
  where $Cl(0, J_{\inte})$ is defined in Definition \ref{limit-1}.
   \end{theorem}
  
The notation $Cl (\cdot, \cdot)$ used in Theorem~\ref{thm1} is defined as follows. 
   
   \begin{definition} \label{limit-1} Assume that ${\mathcal N} = \{(0, 0)\}$. Let $\theta_1, \theta_2 \in [L^2(B_1)]^3$. Define $Cl(\theta_1, \theta_2)= (E_0, H_0)$ where $(E_0, H_0)\in [H(\curl, B_1)]^2$ is the unique solution to the system
   \begin{equation}\label{defcase1}
   \begin{cases}
   \nabla\times E_0 = i\omega\mu H_0 + \theta_1 &\text{ in } B_1,\\[6pt]
   \nabla\times H_0 = -i\omega\eps E_0 + \theta_2 &\text{ in } B_1,\\[6pt]
   \nabla\times E_0 \cdot \nu = \nabla\times H_0 \cdot \nu = 0 &\text{ on } \partial B_1. 
   \end{cases}
   \end{equation}
\end{definition}   
   
 \begin{remark} \rm  The existence and the uniqueness of $(E_0, H_0)$ are established  in Lemma \ref{lem:fredholm}.
 \end{remark} 

\begin{remark} In \cite{Weder2}, the conditions 
$$
 \nabla\times E_0 \cdot \nu |_{int} = \nabla\times H_0 \cdot \nu |_{int}= 0
$$
are also imposed on the boundary of the cloaked region. This is different from \cite{GKLU07-1} (see also \cite[page 459]{Lassas}), where the following boundary  conditions are imposed for solutions satisfying some integrability conditions, which are called finite energy solutions in \cite{GKLU07-1}, 
$$
E_0 \times  \nu|_{int} = H_0 \times  \nu|_{int} = 0. 
$$
\end{remark}  
  

  

The novelty of Theorem~\ref{thm1} relies on the fact that no lossy layer is required. The result holds for a general class of pair $(\eps, \mu)$. Applying Theorem~\ref{thm1} to the case where a fixed lossy-layer is used, one obtains that the degree of visibility  is of the order  $\rho^3$ which is better than $\rho^2$ as established previously in \cite{Bao} for the case $J_{\inte} \equiv 0$. In contrast with  \cite{Bao, Ammari13, DLU}, in Theorem~\ref{thm1}, the estimate of visibility  is considered up to the cloaked region and the behavior of the electromagnetic fields are established inside the cloaked region.

\medskip 

We next consider the resonance case. We begin with the compatible case, i.e., \eqref{passive} below holds.

   \begin{theorem}
   \label{thm1.1} 
  Let $\rho \in (0, 1/2)$,  $R_0> 2$, and $J \in [L^2(\mR^3)]^3$ be  such that $\operatorname{supp} J_{\ext} \subset \subset B_{R_0}\setminus B_2$.  Assume that  system (\ref{medium:cloak}) is resonant and the following compatibility condition {\bf holds}: 
    \begin{equation}
   \label{passive}
   \int_{B_1}J_{\inte}\cdot\bar{\bE}\,dx = 0 \quad \mbox{ for all } (\bE, \bH) \in {\mathcal N}. 
   \end{equation}
We have, for all $K\subset \subset \R^3\setminus \bar{B_1}$, 
 \begin{equation}\label{thm1.1-est1}
 \|(F_\rho^{-1}*E_c, F_\rho^{-1}*H_c) - (E, H)\|_{H(\curl, K)} \leq C \Big(\rho^3 \|J_{\ext}\|_{L^2(B_{R_0}\setminus B_2)}+\rho^2\|J_{\inte}\|_{L^2(B_1)} \Big), 
 \end{equation}
   for some positive constant $C$ depending only on $R_{0}, \omega, K, \mu$, and $\eps$. Moreover, 
\begin{equation}\label{thm1.1-CV}
  \lim_{\rho \to 0} (E_c, H_c) =  Cl(0, J_{\inte}) \mbox{  in } [H(\curl, B_1)]^2,  
\end{equation}
where $Cl(0, J_{\inte})$ is defined in Definition \ref{limit-2}.
  \end{theorem}
  
  In Theorem~\ref{thm1.1}, we use the following notion: 
  
\begin{definition} \label{limit-2}  Assume that ${\mathcal N} \neq  \{(0, 0)\}$.  Let $\theta_1, \theta_2 \in [L^2(B_1)]^3$ be such that 
\begin{equation}
\int_{B_1} \big(\theta_2 \cdot\bar{\bE} - \theta_1 \cdot \bar{\bH} \big) \,dx = 0 \quad \mbox{ for all } (\bE, \bH) \in {\mathcal N}. 
\end{equation}
Let  $(E_0, H_0, E^{\perp}, H^{\perp})\in [H_{\loc}(\curl, \R^3)]^2\times \mathcal N^{\perp}$ be the unique solution of the following systems
   \begin{equation}\label{defcase2.1}
  \begin{cases}\nabla\times E_0 = \nabla\times H_0 = 0 & \text{ in } \R^3\setminus B_1,\\[6pt]
  \dive E_0 = \dive H_0 = 0 & \text{ in } \R^3\setminus B_1,\\[6pt]
  \nabla\times E_0 = i\omega\mu H_0 + \theta_1 &\text{ in } B_1,\\[6pt]
  \nabla\times H_0 = -i\omega\eps E_0 + \theta_2 &\text{ in } B_1,
  \end{cases} \quad \mbox{ and } \quad 
  \begin{cases}
  \nabla\times E^{\perp} = i\omega\mu H^{\perp} &\text{ in } B_1,\\[6pt]
  \nabla\times H^{\perp} = -i\omega\eps E^{\perp} &\text{ in } B_1,\\[6pt]
  \eps E^{\perp}\cdot \nu = E_0\cdot \nu|_{\ext}  &\text{ on } \partial B_1,\\[6pt]
  \mu H^{\perp}\cdot \nu = H_0\cdot \nu|_{\ext} &\text{ on } \partial B_1.
  \end{cases}\end{equation}
such that
\[\big|\big(E_0(x), H_0(x)\big)\big| = O(|x|^{-2}) \mbox{ for large $|x|$}.\]
Denote $Cl(\theta_1, \theta_2)$ the restriction of $(E_0, H_0)$ in $B_1$.
\end{definition}


Here and in what follows,  $\mathcal N(D)^\perp$ denotes the orthogonal space of $\mathcal N(D)$ with respect to the standard scalar product in $[L^2(D)]^6$. The uniqueness and the existence of $(E_0, H_0, E^{\perp}, H^{\perp})$ are given in Lemmas~\ref{lem:uniqueness} and  \ref{lem:jump1}. 

\medskip 
  
In Definition~\ref{limit-2}, $(E_0, H_0)$ is determined by a non-local structure \eqref{defcase2.1}. This is new to our knowledge.  

\medskip 

In the incompatible case, we have
   
    \begin{theorem}
   \label{thm1.2}  Let $\rho \in (0, 1/2)$,  $R_0> 2$, and $J \in [L^2(\mR^3)]^3$ be  such that $\operatorname{supp} J_{\ext} \subset \subset B_{R_0}\setminus B_2$.  Assume that system (\ref{medium:cloak}) is resonant and the compatibility condition does {\bf not} hold, i.e.,      \begin{equation}
   \label{N-passive}
   \int_{B_1}J_{\inte}\cdot\bar{\bE}\,dx \neq  0 \quad \mbox{ for some } (\bE, \bH) \in {\mathcal N}. 
   \end{equation}
We have, for all $K\subset \subset \R^3\setminus \bar{B_1}$,
\begin{equation}\label{thm1.2-est1}
(F_\rho^{-1}*E_c, F_\rho^{-1}*H_c) - (E, H)\|_{H(\curl, K)} \leq C \Big(\rho^3 \|J_{\ext}\|_{L^2(B_{R_0}\setminus B_2)}+\rho\|J_{\inte}\|_{L^2(B_1)} \Big)
\end{equation}
and
   \begin{equation} \label{explosion}\liminf \limits_{\rho\rightarrow 0}  \rho\| \big(E_c, H_c \big)\|_{L^2(B_1)} > 0.
   \end{equation}

   \end{theorem}
   
Some comments on Theorems~\ref{thm1.1} and \ref{thm1.2} are in order.  Theorems~\ref{thm1.1} and \ref{thm1.2} imply in particular that cloaking is achieved even in the resonance case. Moreover, without any source in the cloaked region, one can achieve the same degree of visibility as in the non-resonant case considered in Theorem~\ref{thm1}. Nevertheless, the degree of visibility varies and depends on the compatibility of the source inside the cloaked region. More precisely,   the rate of the convergence of $(E_c, H_c) - (E, H)$ outside $\bar B_1$ in the compatible  case is of the order $\rho^2$ which is better than the incompatible resonant case where an estimate of the order $\rho$ is obtained. The rate of the convergence is optimal  and  discussed in Section~\ref{sect-opt}. By  \eqref{explosion}, the energy inside the cloaked region blows up at least with the rate $1/ \rho$ as $\rho \to 0$ in the incompatible case. 

\medskip 


We now describe briefly the ideas of the proofs of Theorems~\ref{thm1}, \ref{thm1.1} and \ref{thm1.2}.
Set 
\begin{equation}\label{def-EHrho}
  (\pbE_{\rho}, \pbH_{\rho}) = (F_{\rho}^{-1}*E_c, \, F_{\rho}^{-1}*  H_c) \quad  \mbox{ in } \R^3. 
\end{equation}
It follows from  a standard  change of variables formula (see, e.g., Lemma ~\ref{pre:cha}) that 
 $(\pbE_\rho, \pbH_\rho) \in [H_{\loc}(\curl, \R^3)]^2$ is the unique (Silver-M\"uller) radiating solution to
   \begin{align}
   \label{equ:aux1}
   \begin{cases}
   \nabla \times \pbE_{\rho} = i\omega \mu_{\rho} \pbH_{\rho}  &\text{ in } \mathbb{R}^3,\\[6pt]
   \nabla \times \pbH_{\rho} = -i\omega \eps_{\rho} \pbE_{\rho} + J_\rho &\text{ in } \mathbb{R}^3,
   \end{cases}
   \end{align}
   where
\begin{equation}\label{epsmurho}
   \big(\eps_{\rho}, \mu_{\rho} \big) = \big( { F_\rho^{-1}}_*\eps_c, {F_\rho^{-1}}_*\mu_c \big) = \left\{\begin{array}{cl} \big(I, I \big) &\text{ in } \mathbb{R}^3\setminus B_{\rho},\\[6pt]
  \big( \rho^{-1}\eps(\cdot / \rho), \rho^{-1} \mu (\cdot/ \rho) \big) &\text{ in } B_{\rho},\end{array}\right.
\end{equation}
   and 
\begin{equation}\label{Jrho}
 J_\rho =  \left\{ \begin{array}{cl}J_{\ext} &\text{ in } \R^3\setminus B_2,\\[6pt] 
\dsp    \rho^{-2}J_{\inte}(\cdot/ \rho) &\text{ in } B_{\rho}, \\[6pt]
   0 & \text{ otherwise}. \end{array}\right.
\end{equation}
We can then derive Theorems~\ref{thm1}, \ref{thm1.1}, and \ref{thm1.2} by studying the difference between $(\pbE_\rho, \pbH_\rho)$ and $(E, H)$ in $\mR^3 \setminus B_1$ and the behavior of $(\pbE_\rho, \pbH_\rho)(\rho \cdot )$ in $B_1$. 
It is well-known that when material parameters inside a small inclusion are bounded from below and above by positive constants, the effect of the small inclusion is small (see, e.g., \cite{VogeliusVolkov, AVV}). Without this assumption, the effect of the inclusion might not be  small (see, e.g.,  \cite{Kohn, HM1}) unless there is an appropriate lossy-layer, see  \cite{Bao, Ammari13, DLU}. In our setting, the boundedness assumption is violated (see \eqref{epsmurho}) and no lossy-layer is used. Nevertheless, the effect of the small inclusion is still small  due to the special structure induced from \eqref{epsmurho}. 

It is worth noting that System \eqref{defcase1}, which involves in  the definition of resonance and non-resonance, and the condition of compatibility \eqref{passive},  appears very naturally in our context. Indeed, note that if $(E_c, H_c)$ is bounded in $[H(\curl, B_1)]^2$, one can check that, up to a subsequence,  $(\rho \pbE_\rho, \rho  \pbH_\rho)(\rho \cdot) = (E_c, H_c)$ converges weakly in $[H(\curl, B_1)]^2$ to $(E_0, H_0)$ which satisfies  system \eqref{defcase1} with $(\theta_1, \theta_2) = (0, J)$. 

\medskip 
The paper is organized as follows. In Section \ref{sec-pre}, we establish some basic facts   and 
 recall some known results related to Maxwell's equations. 
These materials will be used in the proofs of Theorems \ref{thm1}, \ref{thm1.1}, and \ref{thm1.2}.   The proofs of 
Theorems \ref{thm1}, \ref{thm1.1}, and \ref{thm1.2} are given in Section \ref{sec-main}.  Finally, in Section \ref{sect-opt}, we  discuss the optimality of the convergence rate in Theorems \ref{thm1}, \ref{thm1.1}, and \ref{thm1.2}.

 \section{Preliminaries}\label{sec-pre}
   In this section, we establish some basic facts and recall some known results related to Maxwell's equations that will be repeatedly used in the proofs of Theorems \ref{thm1}, \ref{thm1.1}, and \ref{thm1.2}.  In what follows in this section, $D$ denotes a smooth  bounded open subset of $\R^3$ and on its boundary  $\nu$ denotes its normal unit vector directed to the exterior.   We begin with  a variant of the classic Stokes' theorem for an exterior domain.

\begin{lemma}\label{lem:potential}
Assume that $\R^3\setminus D$ is simply connected and let  $u\in H_{\loc}(\curl, \R^3\setminus D)$ be such that 
\begin{equation}\label{lem-P-1}
\nabla\times u = 0 \mbox{ in } \R^3\setminus D \quad \mbox{ and }  |u(x)| =  O(|x|^{-2}) \mbox{ for large } |x|.  
\end{equation}
There exists $\xi \in H^1_{\loc}(\R^3\setminus D)$ such that 
\begin{equation}\label{lem-P-2}
\nabla \xi = u \mbox{ in } \R^3\setminus D \quad \mbox{ and } \quad 
|\xi(x)| = O(|x|^{-1}) \mbox{ for large }  |x|. 
\end{equation}
\end{lemma}

\begin{proof} By \cite[Theorem 2.9]{Girault}, there 
exists $\eta_n \in H^1(B_{n} \setminus D)$ for large $n$ such that 
$$
\nabla \eta_n = u \mbox{ in } B_n \setminus D \quad \mbox{ and } \quad \int_{\partial B_2} \eta_n = 0.  
$$
It follows that, for $m > n$ large, 
$$
\eta_m =  \eta_n \mbox{ in } B_{n} \setminus D. 
$$
Let $\eta$ be the limit of $\eta_n$ as $n \to + \infty$.  
Then   $\eta \in  H^1_{\loc}(\R^3 \setminus D)$ and
    	\[\nabla \eta = u \mbox{ in } \R^3 \setminus D.\]
Fix $x, y \in \R^3$ large enough with $|y| > |x|$ and denote $\hat x = x/ |x|$ and $\hat y = y/ |y|$. Using \eqref{lem-P-1},  we have, by the fundamental theorem of calculus, 
\begin{equation}
|\eta (x) - \eta(y)| \le    \big|\eta( |y| \hat{y}) - \eta( |y| \hat{x}) \big|+ \big|\eta(|y|\hat{x}) - \eta(|x|\hat{x}) \big|  \le  \frac{C}{|y|} + \int_{|x|}^{|y|} \frac{C}{|r|^2} \, d r
\end{equation}   
for some positive constant $C$ independent of $x$ and $y$. It follows that 
\begin{equation}\label{lem-P-4}
|\eta (x) - \eta(y)| \le  \frac{C}{|y|} +  \frac{C}{|x|}. 
\end{equation}   
Hence   $\lim\limits_{|x|\rightarrow \infty}\eta(x)$ exists. Denote this limit  by $\eta_{\infty}$. 
By letting $|y| \to + \infty$  in \eqref{lem-P-4}, we obtain
\[|\eta(x) - \eta_{\infty}| \leq \frac{C}{|x|}, \mbox{ for $|x|$ large enough. }\]
The conclusion follows with $\xi = \eta - \eta_{\infty}$.  
\end{proof}

Let $U$ be a smooth open subset of $\mR^3$. Denote 
\[
 H(\dive, U) := \big\{\phi \in [L^2(U)]^3: \dive \phi \in L^2(U) \big\}. 
 \]
Concerning a free divergent field in a bounded domain, one has the following result which is related to Stokes' theorem,  see, e.g.,  \cite[Theorems 3.4 and  3.6]{Girault}. 

\begin{lemma}
	\label{potential}
	\label{lem:stream}
Assume that $D$ is simply connected and  let $u\in H(\dive, D)$ be such that 
\begin{equation}
\dive u = 0 \mbox{ in } D \quad \mbox{ and } \quad \int_{\Gamma_i }u \cdot \nu = 0 \mbox{ for all connected component $\Gamma_i$ of $\partial D$.}
\end{equation}
 There exists  $\phi \in [H^1(D)]^3$ such that 
	\[\nabla\times \phi = u \text{ in } D \mbox{ and } \dive \phi = 0 \mbox{ in } D.\]
Assume in addition that $u\cdot \nu = 0 \mbox{ on } \partial D$. Then $\phi$ can be chosen such that
\[
\phi \times \nu = 0  \text{ on } \partial D \quad \mbox{ and } \quad \int_{\Gamma_i } \phi \cdot \nu  = 0
 \mbox{ for all connected component $\Gamma_i$ of $\partial D$.} \]
Moreover, such a  $\phi$ is unique and,  for some positive constant $C$, 
	\[\|\phi\|_{H^1(D)}\leq C\|u\|_{L^2(D)}.\]	
\end{lemma}

The following result is  a type of Helmholtz decomposition. It is a variant of   \cite[Corollary 3.4]{Girault}  where $\sigma$ is a positive constant.

\begin{lemma}
	\label{composition}
	Assume that $D$ is simply connected and let $\sigma$ be $3\times 3$ uniformly elliptic matrix-valued function defined in $D$. For any $v\in [L^2(D)]^3$,  there exist $p\in H^1(D)$ and  $\phi \in [H^1(D)]^3$ such that 
\begin{equation}
v = \sigma \nabla p + \nabla\times \phi  \mbox{ in } D, \quad  \dive \phi = 0 \mbox{ in } D \quad  \mbox{ and } \quad \phi\times \nu = 0 \mbox{ on } \partial D.
\end{equation}	
Moreover, 
	\begin{equation}
	\| p\|_{H^1(D)} + \| \phi \|_{H^1(D)} \le C \| v\|_{L^2(D)}.
	\end{equation}
\end{lemma}

\begin{proof} The proof given here is in the  spirit of \cite{Girault} as follows.  By Lax-Milgram's theorem, there exists a unique solution $p \in H^1(D)$ with $\dsp \int_{D} p = 0$ to the equation
	\begin{equation*}
	\int_{D}\sigma \nabla p\cdot \nabla q \, dx = \int_D v\cdot \nabla q\, dx \mbox{ for all } q \in H^1(D). 
	\end{equation*} 
Moreover, 
\begin{equation}\label{lem-Comp-1}
\|p\|_{H^1(D)}\leq C\|v\|_{L^2(D)}.
\end{equation}
Then
\begin{equation}\label{lem-Comp-2}
\dive(v - \sigma \nabla p) = 0  \mbox{ in } D \quad \mbox{ and } \quad  (v - \sigma \nabla p)\cdot \nu = 0 \mbox{ on } \partial D.
\end{equation}
By Lemma \ref{potential}, there  exists $\phi \in [H^1(D)]^3$ such that 
\begin{equation}\label{lem-Comp-3}
\begin{cases}\nabla\times \phi =  v - \sigma \nabla p &\text{ in } D,\\[6pt]
	\dive \phi = 0 & \text{ in } D,\\[6pt]
	\phi \times \nu = 0 & \text{ on } \partial D,
	\end{cases}
\quad \mbox{ and } \quad 
\|\phi\|_{H^1(D)}\leq C\|v-\sigma\nabla p\|_{L^2(D)}.
\end{equation}
Combining \eqref{lem-Comp-1}, \eqref{lem-Comp-2}, and  \eqref{lem-Comp-3}, 
we reach the conclusion for such a pair $(p, \phi)$.
\end{proof}

We next present two  lemmas concerning the uniqueness of the exterior problems for electro-static settings. They are used in  the study of the exterior problems in the low frequency regime, see Lemma \ref{lem:ex-12}. The first one is

\begin{lemma}\label{uniquenessstatic} Assume that $\R^3\setminus D$ is simply connected. Let $u \in H_{\loc}(\curl, \mR^3 \setminus D) \cap H_{\loc}(\dive, \mR^3 \setminus D)$ be such that 
\begin{equation*}
\left\{\begin{array}{cl}
\nabla\times u = 0 &\text{ in } \R^3\setminus D,\\[6pt]
\dive u = 0 &\text{ in } \R^3\setminus D,\\[6pt]
u\cdot \nu = 0 &\text{ on } \partial D,
\end{array} \right. 
\end{equation*}
and 
\begin{equation}\label{US-1}
|u(x)|  = O(|x|^{-2})\mbox{ for large $|x|$}.  
\end{equation}
Then $u = 0$ in $\mR^3 \setminus D$.  
\end{lemma}
\begin{proof}
By Lemma \ref{lem:potential}, there exists $\xi \in H^1_{\loc}(\R^3\setminus D)$ such that 
\begin{equation}
\nabla \xi = u \mbox{ in $\R^3\setminus D$} \quad \mbox{ and } \quad |\xi (x)| = O(|x|^{-1}) \mbox{ for large $|x|$}.
\end{equation}
Since $\dive u = 0$, we have 
\[\Delta \xi = 0 \text{ in } \R^3\setminus D.\]
Since $\nabla \xi \cdot \nu = u\cdot \nu = 0$ on $\partial D$, it follows that $\xi = 0$ in $\R^3 \setminus D$, see, e.g., \cite[Theorem 2.5.15]{Nedelec}.  Therefore,  $u = 0$. 
\end{proof}

The second lemma is

 \begin{lemma}\label{uniquenessstaticE} Assume that $\R^3\setminus D$ is simply connected and $u \in H_{\loc}(\curl, \mR^3 \setminus D) \cap H_{\loc}(\dive, \mR^3 \setminus D)$ is such that 
\begin{equation*}
\left\{\begin{array}{cl}
\nabla\times u = 0 &\text{ in } \R^3\setminus D,\\[6pt]
\dive u = 0 &\text{ in } \R^3\setminus D,\\[6pt]
u\times \nu = 0 &\text{ on } \partial D, 
\end{array} \right. \quad \int_{\Gamma_i} u \cdot \nu = 0 \mbox{  for all connected component $\Gamma_i$ of $\partial D$}, 
\end{equation*}
and 
\begin{equation}\label{US-2}
|u(x)| = O(|x|^{-2}) \mbox{ for large $|x|$.} 
\end{equation}
Then $u = 0$ in $\mR^3 \setminus D$.  
\end{lemma}

\begin{proof}
By Lemma \ref{lem:potential},  there exists $\xi \in H^1_{\loc}(\R^3\setminus D)$, such that 
\begin{equation}
\nabla \xi = u \mbox{ in } \R^3\setminus D
\quad \mbox{ and } \quad   |\xi(x)|  = O(|x|^{-1}) \quad \mbox{ for large } |x|.
\end{equation}
There exists $\psi \in [H^1_{\loc}(\R^3\setminus D)]^3$, such that 
 	\[\nabla\times \psi = u \mbox{ in } \R^3\setminus D.\]
	Fix $\theta \in C^1(\R^3)$ such that  $0 \le \theta \le 1$,  $\theta = 1$ in $B_1$ and $\supp \theta \subset B_2$.  For $r>0$, set $\theta_r (\cdot )  = \theta (\cdot / r)$ in $\R^3$. Let $t >  s> 0$ be large enough (arbitrary) such that $D \subset \subset B_s$. 
Since $u \times \nu = 0$ on $\partial D$, we obtain, by integration by parts, that 
\begin{equation*}
\int_{\R^3\setminus D} \nabla\times (\theta_t \psi)\cdot\nabla(\theta_s \bar \xi) \, dx 
 = -  \int_{\partial D} \theta_t \psi \cdot \nabla (\theta_s \bar  \xi) \times \nu \, ds = -\int_{\partial D} \psi \cdot \bar{u} \times \nu \, ds = 0. 
\end{equation*}
Letting $t \to + \infty$, we derive that 
\begin{equation}\label{lem-U2-p1}
\int_{\R^3\setminus D} u \cdot\nabla(\theta_s \bar \xi)dx = 0.
\end{equation}	
We have 
\begin{equation}\label{lem-U2-p2}
\int_{B_{2s} \setminus B_s} |u| |\xi| |\nabla \theta_s| \, dx \le C |B_{2s} \setminus B_s| s^{-2} s^{-1} s^{-1} \le C s^{-1} \to 0 \mbox{ as } s \to + \infty.  
\end{equation}
Using the fact that 
$$
u \cdot\nabla(\theta_s \bar \xi) = u \big(\theta_s \nabla \bar \xi + \bar \xi \nabla \theta_s) = \theta_s |u|^2 +  u \bar \xi \nabla \theta_s \mbox{ in } \mR^{3} \setminus D, 
$$
and combining  \eqref{lem-U2-p1} and \eqref{lem-U2-p2}, we obtain 
\begin{equation*}
\int_{\R^3\setminus D} |u|^2  \, dx  = 0, 
\end{equation*}
which yields  $u = 0$ in $\R^3\setminus D$.
\end{proof}

\medskip 
The following result is  a consequence of the Stratton - Chu formula. 

\begin{lemma}\label{lem-SC-1}Let $0< k\leq k_0$. Assume that $D \subset\subset B_1$ and  $(E, H) \in \big[H_{\loc}(\curl, \mR^3 \setminus D) \big]^2$ is a  radiating solution to the Maxwell equations
\[
\left\{\begin{array}{cl}
\nabla\times E = ikH & \mbox{ in }  \R^3\setminus \bar{D}, \\[6pt] 
\nabla\times H = -ikE & \mbox{ in } \R^3\setminus \bar{D}.
\end{array}\right.
\]
We have
\begin{equation}\label{SC-1}
	\Big|\big(E(x), H(x)\big) \Big| \le \frac{C}{|x|^2} \big(1 + k|x| \big) \| (E, H)\|_{L^2(B_3 \setminus D)} \mbox{ for } |x| > 3, 
	\end{equation}
for some positive constant $C$ independent of $x$ and $k$. 
\end{lemma}

\begin{proof} Set
$$
G_k(x, y) = \frac{e^{ik|x-y|}}{4\pi|x-y|} \mbox{ for } x, y \in \mR^3, x\neq y.
$$ 
It is known that, see, e.g., \cite[Theorem 6.6 and (6.10)]{Kress},   the following variant of the Stratton-Chu formula holds,  for $x \in \mR^3 \setminus \bar D$,   
\begin{multline}\label{SC-E}
  E(x)=  \nabla_{x}\times \int_{\partial B_2} \nu(y) \times E (y)  G_k (x, y) dy \\[6pt]
    + ik\int_{\partial B_2}\nu(y)\times H(y)G_k(x,y)dy - \nabla_{x}\int_{\partial B_2}\nu(y)\cdot E(y)G_k(x, y)dy.
\end{multline}
Using the facts 
$$
|\nabla G_k(x, y)| \le \frac{C}{|x|^2}  (1 + k |x|) \mbox{ for } y \in \partial B_2, x \in \mR^3 \setminus B_3  
$$
and, since $\Delta E + k^2 E = 0$ in $\mR^3 \setminus D$, 
$$
\| E \|_{L^\infty(\partial B_2)} \le C \| E \|_{L^2(B_3 \setminus D)}, \mbox{ for some positive constant $C$ depending only on $k_0$, }
$$
we derive from \eqref{SC-E} that
	\begin{equation}\label{cor-SC-1}
	|E(x)| \le \frac{C}{|x|^2} \big(1 + k|x| \big) \| (E, H)\|_{L^2(B_3 \setminus D)} \mbox{ for } |x| > 3. 
	\end{equation}
	Similarly, we obtain 
	\begin{equation}\label{cor-SC-2}
	|H(x)| \le \frac{C}{|x|^2} \big(1 + k|x| \big) \| (E, H)\|_{L^2(B_3 \setminus D)} \mbox{ for } |x| > 3.
	\end{equation}
	The conclusion now follows from \eqref{cor-SC-1} and \eqref{cor-SC-2}. 
\end{proof}

We next recall compactness results related to $H(\curl, \cdot)$ and $H(\dive, \cdot)$.

 \begin{lemma} \label{lem:compact}
Let $\epsilon$ be a measurable symmetric uniformly elliptic matrix-valued function defined in $D$. Assume that one of the following two conditions holds 
\begin{enumerate}
\item[i)]  $(u_n)_{n\in \N}\subset H(\curl, D)$ is a bounded sequence in $H(\curl, D)$ such that 
\[\big(\dive(\epsilon u_n) \big)_{n\in \N}\text{ converges in }  H^{-1}(D) \text{ and } \big(u_n\times \nu \big)_{n\in \N} \text{ converges  in } H^{-1/2}(\partial D).\]

\item[ii)] $(u_n)_{n\in \N}\subset H(\curl, D)$ is a bounded sequence in $H(\curl, D)$ such that 
 \[\big(\dive(\epsilon u_n)\big)_{n\in \N} \text{ is bounded in }  L^2(D) \text{ and } \big(\epsilon u_n \cdot \nu \big)_{n\in \N}\text{ converges  in } H^{-1/2}(\partial D).\]
\end{enumerate}

There exists a subsequence of $(u_n)_{n\in \N}$ which converges in $[L^2(D)]^3$.
\end{lemma}

The conclusion of Lemma~\ref{lem:compact} under condition $i)$ is  \cite[Lemma 1]{HM1} and has its roots in \cite{Haddar} and \cite{Costabel}.  The conclusion of Lemma~\ref{lem:compact} under condition $ii)$ can be obtained in  the same way.
These compactness results play a similar role as the compact embedding of $H^1$ into $L^2$ in the acoustic setting and are basic ingredients in our  approach.  

\medskip 

The following trace results related to $H(\curl, \cdot)$ and $H(\dive, \cdot)$ are standard, see, e.g.,  \cite{Alonso, BC, Girault}.  
  
\begin{lemma}\label{lem:trace} Set $\Gamma = \partial D$. We have 
\begin{enumerate}
\item[i)]   
\[
 \|v \times \nu\|_{H^{-1/2}(\dive_{\Gamma},\Gamma)}\leq C\|v\|_{H(\curl, D)} \mbox{ for } v \in H(\curl, D).  
 \] 

\item[ii)] 
\[
 \|v \cdot \nu\|_{H^{-1/2}(\Gamma)}\leq C\|v\|_{H(\dive, D)} \mbox{ for } v \in H(\dive, D). 
\]
\end{enumerate}

Moreover, for any $h\in H^{-1/2}(\dive_{\Gamma}, \partial D)$,  there exists $\phi \in H(\curl, D)$ such that 
\[
\phi\times \nu = h \mbox{ on } \partial D, \mbox{ and } \|\phi\|_{H(\curl, D)}\leq C\|h\|_{H^{-1/2}(\dive_{\Gamma}, \partial D)}.
\]

Here $C$ denotes a positive constant depending only on $D$.
\end{lemma}

Here and in what follows, we  denote 
\begin{equation*}
H^{-1/2}(\dive_\Gamma, \Gamma): = \Big\{ \phi \in [H^{-1/2}(\Gamma)]^3; \; \phi \cdot \nu = 0 \mbox{ and } \dive_\Gamma \phi \in H^{-1/2}(\Gamma) \Big\},
\end{equation*}
\begin{equation*}
\| \phi\|_{H^{-1/2}(\dive_\Gamma, \Gamma)} : = \| \phi\|_{H^{-1/2}(\Gamma)} +  \| \dive_\Gamma \phi\|_{H^{-1/2}(\Gamma)}.
\end{equation*}

\medskip 
We finally  recall the following  change of variables for the Maxwell equations. It is the basic ingredient for cloaking using transformation optics for electromagnetic fields. 
   
   \begin{lemma}
   	\label{pre:cha}
   	Let $D, D'$ be two open bounded connected subsets of $\R^3$ and  $F: D \rightarrow D'$ be a bijective map such that $F\in C^1(\bar{D}), F^{-1} \in C^1(\bar{D}')$. 
   	Let $\eps, \, \mu \in [L^{\infty}(D)]^{3\times 3}$, and $j \in [L^2(D)]^3$. Assume that $(E, H) \in [H(\curl, D)]^2$ is a solution of  the Maxwell equations
   	\begin{equation}
   	\begin{cases}
   	\nabla \times E = i\omega \mu H & \text{ in } D,\\[6pt]
   	\nabla \times H = -i\omega \eps E +j & \text{ in } D.
   	\end{cases}
   	\end{equation}
   	Set, in $D'$,  
   	$$
   	E': = F*E: = (DF^{-T}E)\circ F^{-1} \quad \mbox{ and } \quad H' := F*H: =   (DF^{-T}H)\circ F^{-1}.
   	$$
   	Then $(E', H') \in [H(\curl, D')]^2$ satisfies 
   	\begin{equation}
   	\begin{cases}
   	\nabla \times E' = i\omega\mu' H' &\text{ in } D',\\[6pt]
   	\nabla \times H' = -i\omega\eps' E' +j' & \text{ in } D',
   	\end{cases}
   	\end{equation}
   	where 
   	\[\eps'  := F_* \eps := \frac{DF \eps DF^{T}}{|\det F|}\circ F^{-1}, \quad  \mu' := F_* \mu := \frac{DF \mu DF^{T}}{|\det F|}\circ F^{-1}, \quad \mbox{and} \quad  j' : =F_*j= \frac{DF j}{|\det F|}\circ F^{-1}. \]
 \end{lemma}
   
   \begin{remark}\rm It is worth noting the difference of $F *$ in the definition of $E'$ and $H'$, and $F_*$ in the definition of $\eps', \, \mu',$ and $j'$.
   	
   \end{remark}

 \section{Proofs of the main results}\label{sec-main}

This section is devoted to the proof of Theorems~\ref{thm1}, \ref{thm1.1}, and \ref{thm1.2} and is organized as follows. In the first subsection, we establish various results related to $(\pbE_\rho, \pbH_\rho)$. The proof of Theorem~\ref{thm1} is given in the second subsection and the ones of Theorems~\ref{thm1.1} and \ref{thm1.2} are given in the third subsection. 

 \subsection{Some useful lemmas}  \label{sect-lemma}

In this section, $D\subset B_1$ denotes a smooth open bounded subset of $\R^3$, and $\eps$ and  $\mu$   denote two $3\times 3$ matrices $\eps, \mu$  defined in $D$ which are both real, symmetric,  and uniformly elliptic in $D$. We also assume that 
$D$ and $\mR^3 \setminus D$ are simply connected and  $\eps, \mu$ are piecewise $C^1$.  The following lemma provides the stability of the exterior problem in the  low frequency regime. 

\medskip

\begin{lemma}
\label{lem:ex-12}
Let  $0 < \rho< \rho_0$ and let  $(E_\rho, H_{\rho}) \in [H_{\loc}(\curl, \R^3 \setminus D)]^2$ be a radiating solution to the system
\begin{equation}\label{eq:ex}\begin{cases}
\nabla \times E_{\rho}  = i\rho H_{\rho} & \text{ in } \mathbb{R}^3\setminus D,\\[6pt]
\nabla \times H_{\rho} = -i\rho E_{\rho}  & \text{ in } \mathbb{R}^3\setminus D.
\end{cases}
\end{equation}
We have, for $R>1$, 
\begin{equation}\label{lem:ex-conclusion}
\|(E_{\rho}, H_{\rho})\|_{H(\curl, B_R\setminus D)} \leq C_R \Big( \|E_{\rho}\times \nu\|_{H^{-1/2}(\partial D)}+ \|H_{\rho}\cdot \nu\|_{H^{-1/2}(\partial D)}\Big)
\end{equation}
and 
\begin{equation}\label{lem:ex0-conclusion}
\|(E_{\rho}, H_{\rho})\|_{H(\curl, B_R\setminus D)} \leq C_R\Big(\|E_{\rho} \times \nu \|_{H^{-1/2}(\partial D)}+ \|H_{\rho}\times \nu\|_{H^{-1/2}(\partial D)}\Big),
\end{equation}
for some positive constant $C_R$ depending only on $\rho_0$, $D$, and $R$.  
\end{lemma}

\begin{proof} We begin with the proof of \eqref{lem:ex-conclusion}. 
Since $(E_{\rho}, H_{\rho})$ satisfies (\ref{eq:ex}), it suffices to prove that
\begin{equation}\label{goal}\|(E_{\rho}, H_{\rho})\|_{L^2(B_R\setminus D)} \leq C_R \Big(\|E_{\rho}\times \nu\|_{H^{-1/2}(\partial D)}+ \|H_{\rho}\cdot \nu\|_{H^{-1/2}(\partial D)} \Big)
 \end{equation}
for $R> 3$. Fixing $R> 3$, we prove \eqref{goal}  by contradiction. Suppose that there exist a sequence $(\rho_n)_{n \in \N} \subset (0, \rho_0)$ and  a sequence of radiating solutions $\big((E_n, H_n) \big)_{n \in \N} \subset [H(\curl, \R^3\setminus D)]^2$  of the system
 \begin{equation}\label{2.3} \begin{cases}
\nabla \times E_{n}  = i\rho_n H_{n} & \text{ in } \mathbb{R}^3\setminus D,\\[6pt]
\nabla \times H_{n} = -i\rho_n E_{n}  & \text{ in } \mathbb{R}^3\setminus D,
\end{cases}\end{equation}
such that 
 \begin{equation}\label{2.1}\|(E_n, H_n)\|_{L^2( B_R\setminus D)} = 1 \text{ for } n\in \N,\end{equation}
 and
\begin{equation}\label{2.2} \lim_{n \to 0} \Big( \|E_n\times \nu\|_{H^{-1/2}(\partial D)}+\|H_n\cdot \nu\|_{H^{-1/2}(\partial D)} \Big) = 0. \end{equation}
Without loss of generality, one might assume that $\rho_n \to \rho_\infty$ as $n\rightarrow \infty$ for some $\rho_{\infty}\in [0, \rho_0]$. We only consider the case $\rho_{\infty} = 0$. The case $\rho_{\infty}> 0$ can be proven similarly. 
From \eqref{2.3} and \eqref{2.1}, we have
\begin{equation} \|( E_n, H_n)\|_{H(\curl, B_R\setminus D)}\leq C.\end{equation}
Here and in what follows in this proof, $C$ and $C_r$ denote positive constants independent of $n$. Applying  Lemma~\ref{lem-SC-1}, we have 
\begin{equation} 
\label{2.1bis} \|( E_n, H_n)\|_{H(\curl, B_r \setminus D)}\leq C_r
\end{equation}
for  $r > 3$. Since 
$$
\Delta E_\rho + \rho^2 E_\rho = \Delta H_\rho + \rho^2 H_\rho = 0 \mbox{ in } \R^3 \setminus D, 
$$
it follows from \eqref{2.1bis} that, for $r >3$,  
$$
\|( E_n, H_n)\|_{H^1(B_{r+1} \setminus B_{r-1})}\leq C_r. 
$$
By the trace theory, we have 
$$
\|( E_n, H_n)\|_{H^{1/2}(\partial B_{r})}\leq C_r. 
$$
Since the embedding of $H^{1/2}(\partial B_r)$ into $H^{-1/2}(\partial B_r)$ is compact, by applying i) of Lemma~\ref{lem:compact} to $(E_n)$  and by applying ii) of Lemma~\ref{lem:compact} to $(H_n)$, without loss of generality, one might assume that $(E_n, H_n)$  converges in $[L^2_{\loc}(\R^3 \setminus D)]^6$. 
Moreover,  the limit $(E, H) \in [H_{\loc}(\R^3 \setminus D)]^2$ satisfies \begin{equation}\label{limitEH}
\left\{\begin{array}{cl}
\nabla\times H = 0 &\text{ in } \R^3\setminus D,\\[6pt]
\dive H = 0 &\text{ in } \R^3\setminus D,\\[6pt]
H\cdot \nu = 0 &\text{ on } \partial D,
\end{array} \right.
\quad \mbox{ and } \quad   \left\{\begin{array}{cl}
\nabla\times E = 0 \text{ in } \R^3\setminus D,\\[6pt]
\dive E = 0 \text{ in } \R^3\setminus D,\\[6pt]
E\times \nu = 0 \text{ on } \partial D.
\end{array} \right.
\end{equation}
Applying Lemma \ref{lem-SC-1} to $(E_n, H_n)$ and letting $n \to + \infty$ ($\rho_n \to 0$), we have
\begin{equation}\label{lem-ex-decay}
\big|\big(E(x), H(x) \big) \big| = O(|x|^{-2})\mbox{ for large } |x|. 
\end{equation}
On the other hand,  since $E_n = \dsp  -\frac{1}{i \rho_n} \nabla \times H_n$ in $\mR^3 \setminus D$, we have
\begin{equation}\label{lem-ex-p1}
\int_{ \Gamma_i} E_n \cdot \nu = 0 \mbox{ for all  connected component $\Gamma_i$ of $\partial D$}. 
\end{equation}
Since $(E_n)$ converges to $E$ in $[L^2_{\loc}(\R^3 \setminus D)]^3$ and $\dive E_n = \dive E = 0$ in $\R^3 \setminus D$, it follows that $(E_n)$ converges to $E$ in $H_{\loc}(\dive, \R^3 \setminus D)$. This in turn implies, by \eqref{lem-ex-p1}, 
\begin{equation}\label{lem-ex-p2}
\int_{ \Gamma_i} E \cdot \nu = 0 \mbox{ for all  connected component $\Gamma_i$ of $\partial D$}. 
\end{equation}
Applying Lemma~\ref{uniquenessstatic} to $H$, we derive from  \eqref{limitEH} and  \eqref{lem-ex-decay} that 
\begin{equation}\label{lem-ex-H=0}
H = 0 \mbox{ in } \R^3 \setminus D. 
\end{equation}
Similarly, applying Lemma~\ref{uniquenessstaticE} to $E$, from  \eqref{limitEH}, \eqref{lem-ex-decay}, and \eqref{lem-ex-p2},  we obtain 
\begin{equation}\label{lem-ex-E=0}
E = 0 \mbox{ in } \R^3 \setminus D. \footnote{When $\rho_\infty> 0$, instead of being a solution of \eqref{limitEH},  $(E, H)$ is the radiating solution of \eqref{eq:ex} with $\rho = \rho_\infty$ and $E \times \nu = 0$ on $\partial D$. This also implies  that $(E, H) = (0, 0)$ in $\R^3 \setminus D$.}
\end{equation}
From \eqref{2.1}, \eqref{lem-ex-H=0}, and \eqref{lem-ex-E=0} and the fact that $(E_n, H_n)$ converges to $(E, H)$ in $L^2_{\loc}(\R^3 \setminus D)$, we reach a contradiction. The proof of \eqref{lem:ex-conclusion} is complete.

\medskip 
We next deal with \eqref{lem:ex0-conclusion}.
The proof of \eqref{lem:ex0-conclusion} is similar to the one of \eqref{lem:ex-conclusion}. However, instead of obtaining \eqref{limitEH} and \eqref{lem-ex-p2}, we have 
\begin{equation*}
\left\{\begin{array}{cl}
\nabla\times H = 0 &\text{ in } \R^3\setminus D,\\[6pt]
\dive H = 0 &\text{ in } \R^3\setminus D,\\[6pt]
H \times \nu = 0 &\text{ on } \partial D, 
\end{array} \right.
\quad \mbox{ and } \quad   \left\{\begin{array}{cl}
\nabla\times E = 0 \text{ in } \R^3\setminus D,\\[6pt]
\dive E = 0 \text{ in } \R^3\setminus D,\\[6pt]
E\times \nu = 0 \text{ on } \partial D, 
\end{array} \right.
\end{equation*}
and 
\begin{equation*}
\int_{ \Gamma} H \cdot \nu = \int_{ \Gamma} E \cdot \nu = 0 \mbox{ for all  connected component $\Gamma$ of $\partial D$}. 
\end{equation*}
By the same arguments, we can derive that $(E, H) = (0, 0)$ in $\mR^3$, which also yields a contradiction. The details are left to the reader. \end{proof}

\begin{remark}\label{rem-div} \rm We have
$$
\dive_\Gamma (E_\rho \times \nu) = \nabla \times E_\rho \cdot \nu = i \rho H_\rho \cdot \nu \mbox{ on } \partial D. 
$$
It follows that 
$$
 \|E_{\rho}\times \nu\|_{H^{-1/2}(\dive_\Gamma, \partial D)} \le \|E_{\rho}\times \nu\|_{H^{-1/2}(\partial D)}+ \|H_{\rho}\cdot \nu\|_{H^{-1/2}(\partial D)} \le \frac{1}{\rho} \|E_{\rho}\times \nu\|_{H^{-1/2}(\dive_\Gamma, \partial D)}. 
$$
\end{remark}

\medskip 
The next lemma gives an estimate for solutions of Maxwell's equations in the low frequency regime,  which in turn implies an estimate for the effect of a small inclusion
after a change of variables. 

\begin{lemma}\label{lem-FF}
	Let $0 < \rho < 1/2$, $R>1/2$,  and let $(E_{\rho}, H_{\rho})\in [H_{\loc}(\curl, \R^3\setminus D)]^2$ be a   radiating solution to the system
	\begin{equation}
		\label{exterior}
		\begin{cases}
			\nabla\times E_{\rho} = i\omega \rho H_{\rho} &\mbox{ in } \R^3\setminus D,\\[6pt]
			\nabla\times H_{\rho} = -i\omega \rho E_{\rho}&\mbox{ in } \R^3\setminus D.
		\end{cases}
	\end{equation}
	We have  
	\[
	\Big|\Big(E_{\rho}(x), H_{\rho}(x)\Big)\Big|\leq C\rho^3\|(E_{\rho}, H_{\rho})\|_{L^2(B_2\setminus D)} \quad  \mbox{ for } x \in B_{3R/\rho}\setminus B_{2R/\rho}, 
	\]
	for some constant $C$ depending only  $ R$.
\end{lemma}

\begin{proof} We only deal with small $\rho$, since otherwise the conclusion is just a consequence of Stratton-Chu's formula. We have,  for $x \in \mR^3 \setminus \bar B_1$,  (see \eqref{SC-E}) 
\begin{multline}\label{SC-E-FF}
E_{\rho}(x)=   \int_{\partial B_1} \nabla_x G_k (x, y) \times \big( \nu(y) \times E_{\rho} (y) \big)  dy \\[6pt]
    + i \omega \rho \int_{\partial B_1}\nu(y)\times H_{\rho}(y)G_k(x,y)dy - \int_{\partial B_1}\nu(y)\cdot E_{\rho}(y) \nabla_x G_k(x, y)dy,
\end{multline}
where $k = \omega\rho$.
We claim that 
\begin{equation}\label{FF-1-p1}
\left| \int_{\partial B_1} E_{\rho} \times \nu \right|  \le C \rho \| (E_{\rho}, H_{\rho}) \|_{L^2(B_2\setminus D)},
\end{equation}
and 
\begin{equation}\label{FF-1-p2}
\left| \int_{\partial B_1} H_{\rho} \times \nu \right|  \le C \rho \| (E_{\rho}, H_{\rho}) \|_{L^2(B_2\setminus D)}. 
\end{equation}
Assuming this, we continue the proof. We have 
\begin{equation}\label{FF-2}
\int_{\partial B_1} \nu \cdot E_{\rho}\, ds = \frac{1}{i \omega \rho} \int_{\partial B_1} \nu \cdot \nabla \times H_{\rho} \,ds= 0. 
\end{equation}
Rewrite  \eqref{SC-E-FF} under the form  
\begin{align*}
& E_{\rho}(x)= \\[6pt] 
&   \int_{\partial B_1} \nabla_x G_k (x, 0) \times \big( \nu(y) \times E_{\rho} (y) \big)  dy  +  \int_{\partial B_1} \big(\nabla_x G_k (x, y) -  \nabla_x G_k (x, 0) \big) \times \big( \nu(y) \times E_{\rho} (y) \big)  dy  \\[6pt]
    & + i \omega \rho \int_{\partial B_1}\nu(y)\times H_{\rho}(y) G_k(x,0)dy  + i \omega \rho \int_{\partial B_1}\nu(y)\times H_{\rho}(y) \big( G_k(x,y) - G_k(x,0) \big)dy \\[6pt]
  & - \int_{\partial B_1}\nu(y)\cdot E_{\rho}(y) \nabla_x  G_k(x, 0) dy -  \int_{\partial B_1}\nu(y)\cdot E_{\rho}(y) \big( \nabla_x  G_k(x, y) - \nabla_x G(x, 0)  \big) dy. 
\end{align*}
Using the facts, for $|x| \in (2R/ \rho, 3R / \rho)$ and $y \in \partial B_1$,  
$$
|G_{k}(x, y)  -  G_k(x, 0)|  \le C \rho^2, \quad |\nabla G_{k}(x, y)  - \nabla G_k(x, 0)| \le C \rho^3,
$$
and
$$
\|(E_{\rho}, H_{\rho})\|_{L^2(\partial B_1)} \leq C\|(E_{\rho}, H_{\rho})\|_{L^2(B_2\setminus D)},
$$
we derive from \eqref{FF-1-p1}, \eqref{FF-1-p2}, and  \eqref{FF-2} that
\begin{equation}\label{lem-FF-cl1}
|E_{\rho}(x)| \le C \rho^3 \| (E_{\rho}, H_{\rho}) \|_{L^2(B_2\setminus D)} \mbox{ for }  x \in B_{3R/ \rho} \setminus B_{2R/ \rho}. 
\end{equation}
Similarly, we have 
\begin{equation}\label{lem-FF-cl2}
|H_{\rho}(x)| \le C \rho^3 \| (E_{\rho}, H_{\rho}) \|_{L^2(B_2\setminus D)} \mbox{ for }  x \in B_{3R/ \rho} \setminus B_{2R/ \rho}. 
\end{equation}
The conclusion now follows from \eqref{lem-FF-cl1} and \eqref{lem-FF-cl2}. 

\medskip 
It remains to prove Claims \eqref{FF-1-p1} and \eqref{FF-1-p2}. We only prove \eqref{FF-1-p1}, the proof of \eqref{FF-1-p2} is similar. Let $(\tilde E_{\rho}, \tilde H_{\rho})\in [H(\curl, B_1)]^2$ be the unique solution to the system
\begin{equation}
\label{FF-3}
\begin{cases}
\nabla\times \tilde E_{\rho} = i\omega\rho (1 + i) \tilde H_{\rho} & \mbox{ in } B_1,\\[6pt]
\nabla\times \tilde H_{\rho} = -i\omega\rho (1 + i) \tilde E_{\rho} & \mbox{ in } B_1, \\[6pt]
\tilde E_{\rho}\times \nu = E_{\rho}\times \nu & \mbox{ on } \partial B_1. 
\end{cases}
\end{equation}
The well-posedness of \eqref{FF-3} follows immediately from Lax-Milgram's theorem. From a standard contradiction argument which involves   the  fact that the following  systems
\[
\begin{cases}
\nabla\times E = 0 &\mbox{ in } B_1,\\[6pt]
\dive E = 0 &\mbox{ in } B_1,\\[6pt]
E\times\nu  = 0 &\mbox{ on } \partial B_1,
\end{cases}\quad \mbox{ and } \quad 
\begin{cases}
\nabla\times H = 0 &\mbox{ in } B_1,\\[6pt]
\dive H = 0 &\mbox{ in } B_1,\\[6pt]
H\cdot\nu  = 0 &\mbox{ on } \partial B_1,
\end{cases} 
\]
only have trivial zero solutions, 
 we obtain 
\[ \|(\tilde E_{\rho}, \tilde H_{\rho})\|_{L^2(B_1)} \leq C\big(\|E_{\rho}\times \nu|_{\ext}\|_{H^{-1/2}(\partial B_1)} + \|H_{\rho}\cdot\nu|_{\ext}\|_{H^{-1/2}(\partial B_1)}\big).\]
This implies
\begin{equation}\label{FF-4}\|(\tilde E_{\rho}, \tilde H_{\rho})\|_{L^2(B_1)} \leq C\|(E_{\rho}, H_{\rho})\|_{L^2(B_2\setminus D)}.\end{equation}
Since
\[\left|\int_{\partial B_1}E_{\rho}\times \nu \,ds\right|= \left|\int_{\partial B_1}\tilde E_{\rho}\times \nu \,ds\right| = \left|\int_{B_1}\nabla\times \tilde E_{\rho} \, dx\right| = \left|\int_{B_1} \omega\rho (1 + i) \tilde H_{\rho} dx\right|, \]
Claim \eqref{FF-1-p1} follows from \eqref{FF-4}.

\medskip 
The proof is complete. \end{proof}

\medskip The following compactness result is used in  the proof of Theorems~\ref{thm1}, \ref{thm1.1}, and \ref{thm1.2}. 


\begin{lemma}\label{lem-compact-couple}
Let $\big(E_n, H_n) \big)_n$ be a bounded sequence in $[H(\curl, D)]^2$ and let $\big( (\theta_{1, n}, \theta_{2, n}) \big)_n$ be a convergent sequence in $[L^2(D)]^6$. Assume that 
\begin{equation}\label{lem-CC-EH}
\left\{\begin{array}{cl}
\nabla \times E_n = i \mu H_n + \theta_{1, n} & \mbox{ in } D, \\[6pt]
\nabla \times H_n = - i \eps E_n + \theta_{2, n} & \mbox{ in } D, 
\end{array} \right. 
\end{equation}
and
\begin{equation}\label{lem-CC-EH-bdry}
\big((\nabla  \times E_n \cdot \nu, \nabla \times H_n \cdot \nu ) \big)_n \mbox{ converges in } [H^{-1/2} ( \partial D)]^2.
\end{equation}
Then, up to a  subsequence, $\big( (E_n, H_n) \big)_n$ converges in $[H(\curl, D)]^2$.
\end{lemma}

\begin{remark} \rm A comparison with Lemma~\ref{lem-compact-couple}  is necessary. The difference between Lemma~\ref{lem-compact-couple} and part $i)$ Lemma \ref{lem:compact} is that  the sequence $(E_n \times \nu)_n$ or $(H \times \nu)_n$ is not required to be convergent in $H^{-1/2}(\partial D)$. The difference between Lemma~\ref{lem-compact-couple} and part $ii)$ Lemma \ref{lem:compact} is that  the sequence $\big( \dive (\eps E_n) \big)_n$ or $\big(\dive (\mu H_n) \big)_n$ is not required to be bounded in $L^2(D)$. Nevertheless, in Lemma~\ref{lem-compact-couple},   \eqref{lem-CC-EH} is assumed.
\end{remark}

\begin{proof}
It suffices to prove that,  up to a  subsequence, $\big((E_n, H_n) \big)_n$ converges in $[L^2(D)]^6$. By Lemma \ref{composition}, there exist $(q_n)_{n }\subset H^1(D)$ and  $(\phi_n)_{n}\subset [H^1(D)]^3$ such that, for all $n$,  
\begin{equation}\label{lem-C-C-1}
\eps E_n = \eps \nabla q_n + \nabla\times \phi_n \mbox{ in D}, \quad  \dive \phi_n = 0  \text{ in } D, \quad \mbox{ and } \quad \phi_n \times \nu = 0  \text{ on } \partial D.
\end{equation}
Moreover, we have 
\begin{equation}\label{lem-compact-couple-estimate1}\|q_n\|_{H^1(D)}+\|\phi_n\|_{[H^1(D)]^3}\leq C\|E_n\|_{L^2(D)} \leq C,\end{equation}
 for some positive constant $C$ independent of $n$. 
From \eqref{lem-compact-couple-estimate1}, without loss of generality, one might assume that 
\begin{equation}
\label{lem-compact-couple-converge1}
(q_n)_n \mbox{ and } (\phi_n)_n  \mbox{ converge in $L^2(D)$ and $[L^2(D)]^3$ respectively}. 
\end{equation}
From \eqref{lem-C-C-1} and an integration by parts, we derive that, for all $n$,  
\begin{equation}\label{lem-compact-couple-equation1}
\int_D \eps \nabla q_n \cdot \nabla p \, dx = \int_D \eps E_n \cdot \nabla p\, dx  \mbox{ for } p \in H^1(D).  
\end{equation}
This implies,  by \eqref{lem-CC-EH},  for $m, n \in \N$, 
\begin{align*}
\int_D \eps \nabla (q_n-q_m) \cdot \nabla (\bar{q}_n-\bar{q}_m) \, dx & = \int_D \eps (E_n-E_m) \cdot \nabla (\bar q_n-\bar q_m)\, dx,\\[6pt]
&=i\int_D \Big(\nabla\times\big( H_n-H_m\big) - (\theta_{2,n} - \theta_{2,m})\Big) \cdot \nabla (\bar q_n- \bar q_m)\, dx.\end{align*}
An integration by parts yields
\begin{multline*}
\int_D \eps \nabla (q_n-q_m) \cdot \nabla  (\bar q_n- \bar q_m) \, dx \\=i\int_{\partial D}\nabla\times\big( H_n-H_m\big)\cdot \nu \; \;  (\bar q_n-\bar q_m)\, ds - i\int_D(\theta_{2,n} - \theta_{2,m}) \cdot \nabla (\bar q_n-\bar q_m)\, dx.
\end{multline*}
By \eqref{lem-CC-EH-bdry} and the convergence of $(\theta_{1, n}, \theta_{2, n})$ in $[L^2(D)]^6$,  the LHS of the above identity converges to $0$ as $m, n \to \infty$. Hence, by the ellipticity of $\eps$,  $(\nabla q_n)_{n }$ is a Cauchy sequence and thus converges in $[L^2(D)]^3$. From \eqref{lem-C-C-1}, we have 
\begin{equation*}
\int_D \eps^{-1}\nabla\times (\phi_n - \phi_m)\cdot\nabla\times (\bar \phi_n - \bar \phi_m)\, dx  = \int_D \nabla\times (E_n-E_m) \cdot(\bar\phi_n - \bar\phi_m)\, dx.
\end{equation*}
By the ellipticity of $\eps$ and the convergence of $(\phi_n)$ in $L^2(D)$,  we derive that  $\big(\nabla \times \phi_n\big)_{n} $ is a Cauchy sequence in $[L^2(D)]^3$ and thus converges in $[L^2(D)]^3$. 
Since 
$$
E_n = \nabla q_n + \eps^{-1} \nabla \times \phi_n, 
$$
$(E_n)_n$ converges in $[L^2(D)]^3$. 

\medskip 
Similarly, up to a subsequence,  $(H_n)_n$ converges in $[L^2(D)]^3$. 
\end{proof}

Using Lemma~\ref{lem-compact-couple} and applying the Fredhom theory, one can prove  the well-posedness of $(E_0, H_0)$ in Definitions~\ref{limit-1} and \ref{limit-2}.  
The first result in this direction is

\begin{lemma}
\label{lem:fredholm}
Let $\theta_1, \theta_2 \in [L^2(D)]^3$. The system
\begin{equation}
\begin{cases}
\label{maxwell}
\nabla\times E = i\mu H + \theta_1 & \text{ in } D,\\[6pt]
\nabla\times H = -i\eps E + \theta_2 & \text{ in } D,\\[6pt]
\nabla\times  E\cdot \nu = \nabla\times H\cdot \nu = 0 & \text{ on } \partial D, 
\end{cases}
\end{equation}
has a solution $(E, H)$ in $[H(\curl, D)]^2$ if and only if 
\begin{equation}\label{condition}\int_{D}\theta_2\cdot \bar{\bE}\, dx - \int_{D}\theta_1 \cdot\bar{\bH}\,dx = 0 ~~ \text{ for all } (\bE, \bH) \in \mathcal N(D).\end{equation}
In particular,  system (\ref{maxwell}) has a unique solution $(E, H)\in\mathcal N(D)^{\perp}$ if and only if 
(\ref{condition}) holds.
\end{lemma}


\begin{proof} Lemma~\ref{lem:fredholm} is derived from  the Fredholm theory.  Since $\eps$ and $ \mu$ are uniformly elliptic, by Lemma \ref{composition}, there exist $p_1, p_2 \in H^1(D)$ and $\phi_1, \phi_2 \in [H^1(D)]^3$ such that
\begin{equation}\label{equ:de}\theta_1 = \mu\nabla p_1 + \nabla\times \phi_1,\quad \theta_2 = \eps \nabla p_2 + \nabla\times \phi_2 \mbox{ in }  D, 
\end{equation}
and \begin{equation}\label{equ:de:bou}\nabla\times \phi_1\cdot \nu =\nabla\times \phi_2\cdot \nu = 0 \text{ on } \partial D .\end{equation}
Set  $(E_0, H_0):= (-i\nabla p_2, i\nabla p_1 )$ in $D$. Then $(E_0, H_0)\in [H(\curl, D)]^2$ is a solution to
\begin{equation}
\begin{cases}\label{onesol}
\nabla\times E_0 = i\mu H_0 + \mu\nabla p_1 & \text{ in } D,\\[6pt]
\nabla\times H_0 = -i\eps E_0 + \eps\nabla p_2 & \text{ in } D,\\[6pt]
\nabla\times  E_0\cdot \nu = \nabla\times H_0\cdot \nu = 0 & \text{ on } \partial D. 
\end{cases}
\end{equation} 
We have 
\begin{equation}\label{condd}\int_{D}\eps \nabla p_2\cdot \bar{\bE}\,dx - \int_{D}\mu \nabla p_1\cdot \bar{\bH}\,dx = 0 ~~ \text{ for all } (\bE, \bH) \in \mathcal N(D).\end{equation} 
From \eqref{equ:de}, \eqref{equ:de:bou}, \eqref{onesol}, and \eqref{condd}, by considering $(E- E_0, H - H_0)$ instead of $(E, H)$, one might assume that $(\theta_1, \theta_2)\in H(\dive, D)$,  
\begin{equation}\label{assumption-01} 
\dive (\theta_1) = \dive (\theta_2 )  = 0   \mbox{ in } D \quad \mbox{ and } \quad   \theta_1\cdot \nu = \theta_2 \cdot \nu = 0 \mbox{ on } \partial D.
\end{equation}
This is assumed from now on. 

\medskip

Set
\[\mathbb{V} = \Big\{\varphi \in H(\curl, D): \dive (\eps \varphi) = 0, \ \eps \varphi\cdot \nu = 0 \text{ on } \partial D,\  \nabla\times \varphi\cdot \nu = 0 \text{ on } \partial D \Big\}.\]
Since $\eps$ and $\mu$ are real, symmetric and uniformly elliptic, $\VV$ is a Hilbert space equipped with the scalar product 
\begin{equation}\label{maxwell-S}
<E, \varphi>_{\VV, \VV} = \int_{D}\mu^{-1}\nabla\times E \cdot \nabla\times \bar{\varphi}  \, dx + \int_{D}\eps E \cdot \bar{\varphi} \,  dx \quad  \mbox{for $E, \varphi \in \VV$.}
\end{equation}
Let $A: \VV \to  \VV$ be defined by 
\begin{equation}\label{maxwell-AE}
<AE, \varphi>_{<\VV, \VV>} =  - 2 \int_{D}\eps E\cdot \bar{\varphi}\,dx \text{ for all } \varphi \in \VV.
\end{equation}
Since $\eps$ is symmetric,  one can easily check that $A$ is self-adjoint. Since $\eps$ and $\mu$ are symmetric and uniformly elliptic, by  Lemma~\ref{lem:compact},  $A$ is compact.

Let $g \in \VV$ be such that 
\begin{equation}\label{maxwell-g}
<g, \varphi>_{<\VV, \VV>} =  \int_{D}i\theta_2\cdot \bar{\varphi} + \int_{D}\mu^{-1}\theta_1\cdot\nabla\times \bar{\varphi} \text{ for all } \varphi \in \VV. 
\end{equation}
We claim that 
\begin{multline} \label{maxwell-Fredholm}
\mbox{system (\ref{maxwell}) has a solution in $[H(\curl, D)]^2$} \\[6pt] \mbox{ if and only if the equation } 
u + Au = g \mbox{ in } \VV \mbox{ has a solution in $\VV$} 
\end{multline}
and 
\begin{multline} \label{maxwell-Fredholm-1}
\mbox{ $(E, H)$ is a solution of \eqref{maxwell} if and only if} \\[6pt]
\mbox{  $E + AE = g$ in $\VV$ and $H = - i \mu^{-1} (\nabla \times E -  \theta_1) $.}
\end{multline}
Assuming this, we continue the proof. By \eqref{maxwell-Fredholm}  and the Fredholm theory, see, e.g., \cite[Chapter 6]{Brezis}, system (\ref{maxwell})  has a solution if and only if 
\begin{equation}\label{maxwell-F-1}
\langle g, \varphi \rangle_{\VV, \VV} = 0 \mbox{ for all } \varphi \in \VV \mbox{ such that } \varphi + A \varphi = 0 \mbox{ in } \VV, 
\end{equation} 
since $A$ is self-adjoint. Applying \eqref{maxwell-Fredholm-1} with $g = \theta_1 = \theta_2 = 0$ and using \eqref{maxwell-S}, \eqref{maxwell-AE}, and \eqref{maxwell-g}, we derive that  condition  \eqref{maxwell-F-1} is equivalent to the fact that 
\begin{equation*}
\int_{D}\theta_2\cdot \bar{\bE}\,dx - \int_{D}\theta_1 \cdot\bar{\bH}\,dx = 0 ~~ \text{ for all } (\bE, \bH) \in \mathcal N(D), 
\end{equation*}
which is \eqref{condition}.

\medskip
The rest of the proof is devoted to establishing Claims \eqref{maxwell-Fredholm} and \eqref{maxwell-Fredholm-1}. 
Let $(E, H)\in [H(\curl, D)]^2$ be a solution to (\ref{maxwell}). From \eqref{assumption-01}, we derive   that $E\in \VV$. Fix $\varphi \in \VV$. Then  $\nabla\times \varphi \cdot\nu = 0 \mbox{ on } \partial D$.  By Lemma~\ref{potential}, there exists $\varphi_0\in [H^1(D)]^3$ such that 
\begin{equation}\label{maxwell-varphi0}
\nabla\times \varphi_0 = \nabla\times \varphi \mbox{ in } D,\quad  \dive \varphi_0 = 0 \mbox{ in } D,\quad \mbox{ and } \quad 
\varphi_0\times \nu = 0 \mbox{ on } \partial D.
\end{equation}
Since $\nabla\times (\varphi_0-\varphi) = 0$ and $D$ is simply connected, there exists $\xi\in H^1(D)$ such that 
\begin{equation}\label{maxwell-xi}
\mbox{ $\varphi_0 - \varphi= \nabla\xi$ in $D$.}
\end{equation}
We have, for $\varphi \in \VV$,  
\begin{equation}\label{maxwell-1}
\int_{D}\mu^{-1}\nabla\times E \cdot \nabla\times \bar{\varphi} \, dx=   i \int_{D}H \cdot \nabla \times \bar{\varphi} + \mu^{-1}\theta_1\cdot\nabla\times \bar{\varphi}\,dx. 
\end{equation}
Using \eqref{maxwell-varphi0} and an integration by parts, we obtain  
\begin{equation}\label{maxwell-2}
 \int_{D}H \cdot \nabla \times \bar{\varphi} \, dx  =    \int_{D}H \cdot \nabla \times \bar{\varphi_0} \, d x =   \int_{D}\nabla \times H \cdot\bar{\varphi_0}\,dx. 
\end{equation}
Using \eqref{maxwell-xi} and the fact $\nabla \times H \cdot \nu = 0$ on $\partial D$, we also get, by an integration by parts, 
\begin{equation*}
 \int_{D}\nabla \times H \cdot\bar{\varphi_0}\,dx 
=  \int_{D}\nabla \times H \cdot\bar{\varphi}\,dx. 
\end{equation*}
This implies, by \eqref{maxwell-2}, 
\begin{equation}\label{maxwell-3}
 \int_{D}H \cdot \nabla \times \bar{\varphi} \, dx  =      \int_{D}\nabla \times H \cdot\bar{\varphi}\,dx. 
\end{equation}
A combination of \eqref{maxwell-1} and  \eqref{maxwell-3} yields 
\begin{equation}\label{maxwell-4}
\int_{D}\mu^{-1}\nabla\times E \cdot \nabla\times \bar{\varphi} \, dx=   i \int_{D} \nabla \times H  \cdot \bar{\varphi} + \mu^{-1}\theta_1\cdot\nabla\times \bar{\varphi}\,dx. 
\end{equation}
We derive from \eqref{maxwell} and \eqref{maxwell-4} that 
\begin{equation}\label{weak}
\int_{D}\mu^{-1}\nabla\times E \cdot \nabla\times \bar{\varphi} \, dx =  \int_{D} \eps E \cdot\bar{\varphi}\,dx +i\int_{D} \theta_2\cdot \bar{\varphi}\,dx +  \int_{D}\mu^{-1}\theta_1\cdot\nabla\times \bar{\varphi}\,dx.
\end{equation}
It follows  from  \eqref{maxwell-S}, \eqref{maxwell-AE}, and \eqref{maxwell-g} that 
$$
E + A E = g \mbox{ in } \VV. 
$$ 

\medskip
Conversely, assume that there exists $u\in \VV$ such that $u + A u  = g$. Set 
$$
E = u \mbox{ and }  H = - i \mu^{-1} (\nabla \times E - \theta_1) \mbox{ in } D. 
$$ 
Using  \eqref{weak}, one can check that $(E, H)$ satisfies the first two equations of \eqref{maxwell}. It is clear that $\nabla\times  E \cdot \nu = 0$ on $\partial D$ by the definition of $\VV$. Since $\nabla \times H = - i \eps E + \theta_2$ in $D$, $\eps E \cdot \nu = 0$ on $\partial D$ ($E \in \VV$), and  $\theta_2 \cdot \nu = 0 $ on $\partial D$ by \eqref{assumption-01}, we obtain 
$$
\nabla \times H \cdot \nu = 0 \mbox{ on } \partial D. 
$$
The proof is complete. 
\end{proof}

\begin{remark} \rm One of the  key points in the proof of Lemma~\ref{lem:fredholm} is the identity
$$
\int_{D} H  \cdot \nabla \times \bar E \, dx  = \int_{D} \nabla \times H \cdot  \bar E \, dx, 
$$
if $E, H \in H(\curl, D)$ is such that  $\nabla \times E \cdot \nu = \nabla \times H \cdot \nu = 0$ on $\partial D$, see \eqref{maxwell-3}. This ensures the variational character of 
system \eqref{maxwell}.
\end{remark}

The following lemma yields the uniqueness of $(E_0, H_0)$ in Definition~\ref{limit-2}. 
 \begin{lemma}
   \label{lem:uniqueness}
  Let $[(E, H) ,(\tilde{E}, \tilde{H})]\in [H_{\loc}(\curl, \R^3)]^2\times \mathcal N(D)^{\perp}$ be such that 
  \begin{equation}\label{equ:mix}
  \begin{cases}\nabla\times E = \nabla\times H = 0 & \text{ in } \R^3\setminus D,\\[6pt]
  \dive E = \dive H = 0 & \text{ in } \R^3\setminus D,\\[6pt]
  \nabla\times E = i\mu H &\text{ in } D,\\[6pt]
  \nabla\times H = -i\eps E &\text{ in } D,
  \end{cases} \quad  \mbox{ and }  \quad 
  \begin{cases}
  \nabla\times \tilde{E} = i\mu \tilde{H} &\text{ in } D,\\[6pt]
  \nabla\times \tilde{H} = -i\eps \tilde{E} &\text{ in } D,\\[6pt]
  \eps \tilde{E}\cdot \nu = E\cdot \nu|_{\ext}  &\text{ on } \partial D,\\[6pt]
  \mu \tilde{H}\cdot \nu = H\cdot \nu|_{\ext} &\text{ on } \partial D,
  \end{cases}\end{equation}
and
\begin{equation}\label{decaymix} \Big|\big(E(x), H(x)\big)\Big| = O(|x|^{-2})\mbox{ for large } |x|. 
\end{equation} 
Then $(E, H) = (0, 0) \mbox{ in } \R^3$ and $(\tilde{E}, \tilde{H}) = (0, 0) \mbox{ in } D$. 
\end{lemma}

\begin{proof}
Applying Lemma~\ref{lem:potential} to $\bar{E}$, there exists a function $\theta \in H^1_{\loc}(\R^3\setminus D)$ such that 
 \begin{equation}\label{decaytheta}
 \mbox{ $\nabla \theta = \bar{E}$ in $\R^3 \setminus D$ and } |\theta(x)| = O(|x|^{-1})\mbox{ for large } |x|.
 \end{equation}
 For $R>0$ large, since $\dive E = 0 \mbox{ in } \R^3\setminus D$, we have
 \begin{equation*}
 \int_{B_R\setminus D}|E|^2dx   =  \int_{B_R\setminus D}E\cdot\nabla \theta dx = \int_{\partial B_R}(E\cdot \nu) \theta ds - \int_{\partial D}(E\cdot \nu)  \big|_{\ext} \theta ds. 
 \end{equation*}
Letting $R$ tend to $+ \infty$ and using \eqref{decaymix} and \eqref{decaytheta}, we obtain
 \begin{equation}\label{lem-uniqueness-p1}
 \int_{\R^3\setminus D}|E|^2\, dx = -\int_{\partial  D}(E\cdot \nu) \big|_{\ext} \theta \, ds.
 \end{equation}
Extend $\theta$ in $D$ so that the extension belongs to $H^1_{\loc}(\mR^3)$ and  still denote this extension by $\theta$. 
 We derive from the system of $(\tilde E, \tilde H)$ in  \eqref{equ:mix} that 
 \begin{align} \label{lem-uniqueness-p2}
 -\int_{\partial  D}(E\cdot \nu)|_{\ext} \theta \, ds = &  -\int_{\partial  D}(\eps \tilde{E}\cdot \nu) \theta \, ds =   -\int_{D}\eps \tilde{E}\cdot\nabla \theta \, dx - \int_D \dive (\eps \tilde E) \theta  \, dx \nonumber \\[6pt]
 = & \int_{D}-i\nabla\times \tilde{H}\cdot\nabla \theta \, dx = -i\int_{\partial D}\tilde{H}\cdot(\nabla  \theta \times \nu)\, ds  = -i\int_{\partial D}\tilde{H}\cdot(\bar{E} \times \nu)\, ds.  \end{align}
Combining \eqref{lem-uniqueness-p1} and \eqref{lem-uniqueness-p2} yields 
 \begin{equation}\label{lem-uniqueness-p3}
 \int_{\R^3\setminus D}|E|^2\, dx =  -i\int_{\partial D}\tilde{H}\cdot(\bar{E} \times \nu)\, ds. 
 \end{equation}
Similarly, we have 
 \begin{equation}\label{lem-uniqueness-p4}
 \int_{\R^3\setminus D}|H|^2\, dx =  i\int_{\partial D}\tilde{E}\cdot(\bar{H} \times \nu)\, ds. 
 \end{equation}
An integration by parts implies 
\begin{multline*}
\int_{\partial D}\tilde{H}\cdot(\bar{E} \times \nu) \, ds - \int_{\partial D}\tilde{E}\cdot(\bar{H} \times \nu)\, ds \\[6pt]
= \int_{D} \nabla \times  \tilde H \cdot \bar E \, dx - \int_{D} \nabla  \times \bar E \cdot \tilde H \, dx - \int_{D} \nabla \times  \tilde E \cdot \bar H \, dx + \int_{D} \nabla \times \bar H \cdot \tilde E \, dx.
\end{multline*}
Using the equations of $(E, H)$ and $(\tilde E, \tilde H)$ in $D$ in \eqref{equ:mix}, we obtain 
\begin{equation}\label{lem-uniqueness-p5}
\int_{\partial D}\tilde{H}\cdot(\bar{E} \times \nu)\, ds - \int_{\partial D}\tilde{E}\cdot(\bar{H} \times \nu)\, ds =0. 
\end{equation}
A combination of \eqref{lem-uniqueness-p3},  \eqref{lem-uniqueness-p4},  and  \eqref{lem-uniqueness-p5} yields 
\[\int_{\R^3\setminus D}\big(|E|^2+ |H|^2\big)dx = 0.\] 
We derive that  $E = H = 0$ in $\R^3\setminus D$.  This implies, by the unique continuation principle see, e.g., \cite[Theorem 1]{Nguyen}, 
$$ 
E = H = 0 \mbox{ in } D 
$$ 
and,  since $(\tilde E, \tilde H) \in {\mathcal N(D)}^\perp$,  
$$
 \tilde{E} = \tilde{H} = 0  \mbox{ in } D.
$$ 
The proof is complete.
 \end{proof}

\subsection{Approximate cloaking in  the non-resonant case - Proof of Theorem \ref{thm1}}

The key ingredient in the proof of Theorem~\ref{thm1} is the following lemma whose proof uses various results in Section~\ref{sec-pre} and Section~\ref{sect-lemma}

 \begin{lemma}
 	\label{lem:jump1} Let $0 < \rho < \rho_0$, $\theta_\rho = (\theta_{1, \rho}, \theta_{2, \rho}) \in [L^2(D)]^6$, and $h_\rho = (h_{1, \rho}, h_{2, \rho}) \in [H^{-1/2}(\dive_{\Gamma}, \partial D)]^2$. Let $(E_{\rho},H_{\rho})\in [\bigcap_{R > 1}H(\curl, B_R \setminus \partial D)]^2$ be the unique  radiating solution to the system
 	\[
 	\begin{cases}
 	\nabla\times E_{\rho} = i\rho H_{\rho} & \text{ in } \R^3\setminus D,\\[6pt]
 	\nabla\times H_{\rho} = -i\rho E_{\rho} & \text{ in } \R^3\setminus D,\\[6pt]
 	\nabla\times E_{\rho} = i\mu H_{\rho} + \theta_{1, \rho} & \text{ in } D,\\[6pt]
 	\nabla\times H_{\rho} = -i\eps E_{\rho} + \theta_{2, \rho} & \text{ in } D,\\[6pt]
 	[E_{\rho}\times \nu] = h_{1, \rho},  [H_{\rho}\times \nu] = h_{2, \rho} & \text{ on } \partial D. 
 	\end{cases}
 	\]
Assume that $\mathcal N(D) = \{(0, 0)\}$. We have
 	\begin{equation}\label{jump-eq}\|(E_{\rho}, H_{\rho})\|_{L^2(B_5)}\leq C\Big(\| \theta_\rho \|_{L^2(D)} + \|h_\rho \|_{H^{-1/2}(\dive_{\Gamma}, \partial D)}\Big),\end{equation}
 	for some positive constant $C$ depending only on $\rho_0, \,  \eps,  \, \mu$. Assume in addition that 
 	\[ \lim_{\rho \to 0 } \| h_\rho \|_{H^{-1/2}(\dive_{\Gamma}, \partial D)}  = 0 \quad \mbox{ and }  \quad 
  \lim_{\rho \to 0} \theta_\rho =  \theta \mbox{ in } [L^2(D)]^6, \]
for some $\theta = (\theta_1, \theta_2) \in [L^2(D)]^6$. 
 	We have 
 	\begin{equation}\label{jump-convergence}
 	\lim_{\rho \to 0 }(E_{\rho}, H_{\rho}) =  Cl(\theta_1, \theta_2) \mbox{ in } [H(\curl, D)]^2.
 	\end{equation}
 \end{lemma}
 
 Here and in what follows on $\partial D$, $[u]$ denotes the jump of $u$ across $\partial D$ for an appropriate (vectorial) function $u$, i.e., $[u] = u|_{\ext} - u|_{\inte}$ on $\partial D$.

 \begin{proof}  By Lemma~\ref{lem:trace},  without loss of generality, one might assume that $h_{1, \rho} = h_{2, \rho} = 0$ on $\partial D$. This is assumed from now on.  
 
 \medskip 
 We first prove \eqref{jump-eq}  by contradiction. Assume that there exist  sequences $(\rho_n)_{n} \subset (0, \rho_0)$, $\big((E_n, H_n)\big)_{n}\subset [H_{\loc}(\curl, \R^3)]^2$, $\big((\theta_{1, n}, \theta_{2, n})\big)_{n}\subset [L^2(D)]^6$ such that 
 	\begin{equation}\label{equ:contra:main}
 	\begin{cases}
 	\nabla\times E_{n} = i\rho_{n} H_n & \text{ in } \R^3\setminus D,\\[6pt]
 	\nabla\times H_{n} = -i\rho_{n} E_n & \text{ in } \R^3\setminus D,\\[6pt]
 	\nabla\times E_{n} = i\mu H_n + \theta_{1, n} & \text{ in } D,\\[6pt]
 	\nabla\times H_{n} = -i\eps E_n + \theta_{2, n} & \text{ in } D,
 	\end{cases}
 	\end{equation}
 	\begin{equation}\label{unit-jump}   \|(E_n, H_n)\|_{L^2(B_5)} = 1 \text{ for all } n\in \mathbb{N}, \end{equation}
 	and
 	\begin{equation}\label{dive-jump}  \lim_{n \to + \infty }\|(\theta_{1, n}, \theta_{2, n})\|_{L^2(D)}= 0.
	\end{equation}
%
 	Without loss of generality, one might assume that $\rho_n \rightarrow \rho_{\infty}\in [0, \rho_0]$. We only consider the case $\rho_{\infty} = 0$. The case $\rho_{\infty} > 0$ can be proved similarly.  
 	
 	We have 
 	\begin{equation}\label{lem:jump-p1}
 	\nabla \times E_n  \cdot \nu \big|_{\inte}=  \nabla  \times  E_n \cdot \nu \big|_{\ext} = i \rho_n H_n \cdot \nu \big|_{\ext} \to 0 \mbox{ in } H^{-1/2}(\partial D) \mbox{ as } n \to \infty. 
 	\end{equation}
 	Similarly, we obtain 
 	\begin{equation}\label{lem:jump-p2}
 	\nabla \times H_n  \cdot \nu \big|_{\inte} \to 0 \mbox{ in } H^{-1/2}(\partial D) \mbox{ as } n \to \infty. 
 	\end{equation}
 	Applying Lemma~\ref{lem-compact-couple} to $\big((E_n, H_n) \big)_n$ in $D$, without loss of generality, one might assume that 
 	\begin{equation}\label{lem:jump-strong-convergence}
 	\big( (E_n, H_n) \big)_n \mbox{ converges in } [H(\curl, D)]^2 \mbox{ as } n \to \infty. 
 	\end{equation}
 	Applying i) of  Lemma \ref{lem:trace},  we derive that 
 	$$
 	\big( (E_n\times \nu , H_n\times \nu) \big)_n \mbox{ converges in } [H^{-1/2}(\dive_\Gamma, \partial D)]^2 \mbox{ as } n \to \infty. 
 	$$
 	It follows from \eqref{unit-jump}, Lemma~ \ref{lem-SC-1}, and i) of Lemma \ref{lem:compact} that  
 	\begin{equation}\label{lem:jump-strong-convergence-1}
 	\big((E_n, H_n) \big)_n \mbox{ converges in $[L^2_{\loc}(\mR^3 \setminus D)]^6$} \mbox{ as } n \to \infty. 
 	\end{equation}
	
 	Let $(E, H)$ be the limit of $(E_n, H_n)$ in $[L^2_{\loc}(\mR^3)]^6$. Then $(E, H)\in [H_{\loc}(\curl, \R^3)]^2$ and \footnote{In the case $\rho_{\infty} > 0$, the limit $(E, H)$ satisfies the radiating condition and is a solution to Maxwell equations in $\R^3$ with vanished data. It follows that $(E, H) = (0, 0)$, which also gives a  contradiction.}
 	\begin{equation}\label{equ:contra:limit}
 	\begin{cases}
 	\nabla\times E = \nabla\times H = 0  & \text{ in } \R^3\setminus D,\\[6pt]
 	\dive E = \dive H = 0 & \text{ in }  \R^3\setminus D,\\[6pt]
 	\nabla\times E = i\mu H  & \text{ in } D,\\[6pt]
 	\nabla\times H = -i\eps E  & \text{ in } D.
 	\end{cases}
 	\end{equation}
 	We derive from \eqref{lem:jump-p1} and \eqref{lem:jump-p2} that 
 	\begin{equation}\label{normal-jump-3}
 	\nabla\times E\cdot \nu|_{\inte}  = \nabla\times H\cdot \nu|_{\inte}  = 0 \mbox{ on } \partial D.
 	\end{equation}
 	Applying Lemma \ref{lem-SC-1}, we have  
 	\begin{equation}\label{decay-jump}
 	|\big(E(x), H(x)\big)| \leq \frac{C}{|x|^2} \mbox{ for } |x| > 3,
 	\end{equation} 
 	for some positive constant $C$.  Combining \eqref{equ:contra:limit} and \eqref{normal-jump-3} yields  that $(E, H) \big|_{D} \in \mathcal N(D)$. Since $\mathcal N(D) = \{(0, 0)\}$, it follows that $E = H = 0$ in $D$. Hence 
 	\begin{equation}\label{normal-jump-4}
 	E \times  \nu  =  H \times  \nu  = 0 \mbox{ on } \partial D.
 	\end{equation}
 	We have, for each connected component $\Gamma$ of $\partial D$,  
 	$$
 	\int_{\Gamma} E \cdot \nu|_{\ext} = \lim_{n \to \infty}\int_{\Gamma} E_n|_{\ext} \cdot \nu = \lim_{n \to \infty} \frac{1}{- i \rho_n }\int_{\Gamma} (\nabla \times H_n)  \cdot \nu|_{\ext} =0 
 	$$
 	and similarly 
 	$$
 	\int_{\Gamma} H \cdot \nu|_{\ext}  = 0.  
 	$$
 	Using  \eqref{equ:contra:limit}, \eqref{decay-jump}, and \eqref{normal-jump-4},  and applying Lemma 
 	\ref{uniquenessstaticE} to $(E, H)$ in $\mR^3 \setminus D$, we obtain 
 	$$
 	E = H = 0 \mbox{ in } \mR^3 \setminus D. 
 	$$
 	Thus $E = H = 0$ in $\R^3$, which, by using \eqref{lem:jump-strong-convergence} and \eqref{lem:jump-strong-convergence-1}, contradicts  \eqref{unit-jump}. Therefore, \eqref{jump-eq} is proved. 
	
	\medskip 
 	
 	We next establish \eqref{jump-convergence}. Fix an arbitrary sequence $(\rho_n)_n$ converging to $0$. From \eqref{jump-eq}, one obtains that 
 	\[\|\big(E_{\rho_n}, H_{\rho_n}\big)\|_{L^2(B_5)}\leq C\Big(\|\theta_{\rho_n}\|_{L^2(D)}+\|h_{\rho_n}\|_{H^{-1/2}(\dive_{\tau}, \partial D)}\Big) \leq C.\]
 	Using the same argument as above, one obtains that, up to a subsequence, $(E_{\rho_n}, H_{\rho_n})$
 	converges in  $[H(\curl, \mR^3)]^2$ to $(E, H)$, the unique solution of 
 	\begin{equation}\label{equ:res-contra:limit1}
 	\begin{cases}
 	\nabla\times E = \nabla\times H = 0  & \text{ in } \R^3\setminus D,\\[6pt]
 	\dive E = \dive H = 0 & \text{ in }  \R^3\setminus D,\\[6pt]
 	\nabla\times E = i\mu H  + \theta_1 & \text{ in } D,\\[6pt]
 	\nabla\times H = -i\eps E + \theta_2 & \text{ in } D.
 	\end{cases}
 	\end{equation}
 	This system implies $\nabla\times E\cdot \nu|_{\inte} = \nabla\times H \cdot \nu|_{\inte} = 0$ on $\partial D$. Since $\mathcal N(D) = \{(0, 0)\}$, we have  $(E, H)|_{D} = Cl(\theta_1, \theta_2)$. Since $(\rho_n) \to 0$ arbitrarily, assertion \eqref{jump-convergence} follows. The proof is complete.
 \end{proof}

We are ready to give the 

\medskip
\noindent{\bf Proof of Theorem~\ref{thm1}.} Let  $(E_{1, \rho}, H_{1, \rho}) \in [H_{\loc}(\curl, \R^3 \setminus B_{\rho})]^2$ be the unique    radiating solution to the system
\begin{equation}
\label{equ:aux2}
\begin{cases}
\nabla \times E_{1, \rho} = i\omega H_{1, \rho}  &\text{ in } \mathbb{R}^3\setminus B_{\rho},\\[6pt]
\nabla \times H_{1, \rho} = -i\omega E_{1, \rho} + J_{\ext} & \text{ in } \mathbb{R}^3\setminus B_{\rho},\\[6pt]
E_{1, \rho} \times \nu= 0 &\text{ on } \partial B_{\rho}, 
\end{cases}\end{equation}
extend $(E_{1, \rho}, H_{1, \rho})$ by $(0, 0)$ in $B_{\rho}$, and still denote this extension by $(E_{1, \rho}, H_{1, \rho})$. Define
\[
(E_{2, \rho}, H_{2, \rho}) := (E, H) - (E_{1, \rho}, H_{1, \rho})\quad \mbox{ and } \quad  (E_{3, \rho}, H_{3, \rho}): = (\pbE_{\rho}, \pbH_{\rho}) - (E_{1, \rho}, H_{1, \rho})  \quad \mbox{ in } \mR^3.
\]
Then  $(E_{2, \rho}, H_{2, \rho})\in [H_{\loc}(\curl, \R^3 \setminus B_{\rho})]^2$ is the unique  radiating solution to the system
\begin{equation*}
\begin{cases}
\nabla \times E_{2, \rho} = i\omega H_{2, \rho}  &\text{ in } \mathbb{R}^3\setminus B_{\rho},\\[6pt]
\nabla \times H_{2, \rho} = -i\omega E_{2, \rho} & \text{ in } \mathbb{R}^3\setminus B_{\rho},\\[6pt]
E_{2, \rho} \times \nu= E\times \nu &\text{ on } \partial B_{\rho},
\end{cases}
\end{equation*}
and $(E_{3, \rho}, H_{3, \rho})\in [\bigcap_{R> 1}H(\curl, B_R \setminus \partial B_{\rho})]^2$ is the unique  radiating solution to the system 
\begin{equation}\label{continue3-???}
\begin{cases}
\nabla \times E_{3, \rho} = i\omega \mu_\rho H_{3, \rho}  &\text{ in } \mathbb{R}^3\setminus \partial B_\rho,\\[6pt]
\nabla \times H_{3, \rho} = -i\omega \eps_\rho E_{3, \rho} + J_{\rho} \chi_{B_{\rho}} & \text{ in } \mathbb{R}^3\setminus \partial B_\rho,\\[6pt]
[E_{3, \rho} \times \nu] = 0,  \, [H_{3, \rho} \times \nu] = -H_{1, \rho}\times \nu |_{\ext}  &\text{ on } \partial B_\rho, 
\end{cases}
\end{equation}
where $\chi_D$ denotes the characteristic function of a subset $D$ of $\R^3$.  
Recall that $J_\rho$ is defined in  \eqref{Jrho}.  Set 
\[
\tilde E_{2, \rho}(x) = E_{\rho}(\rho x) \mbox{ and } \tilde H_{2, \rho}(x) = H_{\rho}(\rho x) \mbox{ for }x \in \R^3\setminus B_1.
\]
Then $(\tilde E_{2, \rho}, \tilde H_{2, \rho}) \in [H(\curl, \R^3\setminus B_1)]^2$ is the 
unique radiating solution to the system
\begin{equation}
\begin{cases}
\nabla \times \tilde E_{2, \rho} = i\omega\rho \tilde H_{2, \rho}  &\text{ in } \mathbb{R}^3\setminus B_{1},\\[6pt]
\nabla \times \tilde H_{2, \rho} = -i\omega\rho \tilde E_{2, \rho} & \text{ in } \mathbb{R}^3\setminus B_1,\\[6pt]
\tilde E_{2, \rho} \times \nu= E(\rho \, \cdot )\times \nu &\text{ on } \partial B_1.
\end{cases}
\end{equation}
By Lemmas \ref{lem:ex-12} and \ref{lem-FF} (also Remark~\ref{rem-div}), we have, for $R > 1/2$ and for $x\in B_{3R}\setminus B_{2R}$, 
\begin{align*}
\left|\Big(\tilde E_{2, \rho}\Big(\frac{x}{\rho}\Big), \tilde H_{2, \rho}\Big(\frac{x}{\rho}\Big)\Big)\right|&\leq C\rho^3\|(\tilde E_{2, \rho} , \tilde H_{2, \rho})\|_{L^2(B_2\setminus B_1)}\\[6pt]
& \leq C\rho^3(\|E(\rho .)\times \nu\|_{H^{-1/2}(\partial B_1)} + \rho^{-1}\|\dive_{\partial B_1} (E(\rho .)\times\nu)\|_{H^{-1/2}(\partial B_1)})\\[6pt]
& \leq C\rho^3(\|E(\rho .)\times\nu\|_{H^{-1/2}(\partial B_1)} +\|H(\rho .)\cdot\nu\|_{H^{-1/2}(\partial B_1)}). 
\end{align*}
Here and in what follows in this proof, $C$ denotes a positive constant depending only on  $\rho_0$, $R_0$, and $R$.  It follows from the definition of $(\tilde E_{2, \rho}, \tilde H_{2, \rho})$ that
\begin{equation}
\label{inequ1}
\|(E_{2, \rho}, H_{2, \rho})\|_{L^2(B_{3R}\setminus B_{2R})}\leq C\rho^3\|J_{\ext}\|_{L^2(\R^3\setminus B_2)}.
\end{equation}

From now on in this proof, for any vector field $v$, we denote \footnote{With this notation, one has $(E_c, H_c)(x) = (\hat \pbE_\rho, \hat \pbH_\rho)$ in $B_1$. It is worth noting that $\hat{v}(\cdot) \neq v(\rho \,  \cdot)$.}
\begin{equation}\label{hatvv}
\hat{v}(\cdot):= \rho v(\rho \,  \cdot). 
\end{equation}
We claim that
\begin{equation}\label{equ:es-1}
  \|\hat{H}_{1, \rho}\times \nu|_{\ext} \|_{H^{-1/2}(\partial B_1)} + \|\hat{E}_{1, \rho}\cdot \nu|_{\ext} \|_{H^{-1/2}(\partial B_1)}  \le C \rho \| J_{\ext} \|_{L^2(\R^3\setminus B_2)}
 \end{equation}
and, for $R > 1/2$, 
\begin{equation}
\label{inequ2}
\|(E_{3, \rho}, H_{3, \rho})\|_{L^2(B_{3R}\setminus B_{2R})}\leq C\big(\rho^3\|J_{\ext}\|_{L^2(\R^3\setminus B_2)} + \rho^2 \|J_{\inte}\|_{L^2(B_1)} \big).
\end{equation}
It is clear that  \eqref{thm1-est1}  follows from  (\ref{inequ1}) and  (\ref{inequ2}). Moreover, 
by Lemma~\ref{lem:jump1}, assertion  \eqref{thm1.1-CV} now follows from \eqref{equ:es-1}  and the fact that $(E_c, H_c) = (\hat E_{3, \rho}, \hat H_{3, \rho})$ in $B_1$. 

\medskip

It remains to establish  \eqref{equ:es-1} and (\ref{inequ2}).
It is clear that  $(\hat{E}_{3, \rho}, \hat{H}_{3, \rho})\in [\bigcap_{R> 0} H(\curl, B_R\setminus \partial B_1)]^2$ is the unique radiating solution to the system
\begin{equation}\label{continue3}
\begin{cases}
\nabla \times \hat{E}_{3, \rho} = i\omega\rho \hat{H}_{3, \rho}  &\text{ in } \mathbb{R}^3\setminus B_1,\\[6pt]
\nabla \times \hat{H}_{3, \rho} = -i\omega\rho \hat{E}_{3, \rho} & \text{ in } \mathbb{R}^3\setminus B_1,\\[6pt]
\nabla \times \hat{E}_{3, \rho} = i\omega\mu \hat{H}_{3, \rho}  &\text{ in }  B_1,\\[6pt]
\nabla \times \hat{H}_{3, \rho} = -i\omega\eps \hat{E}_{3, \rho}  + J_{\inte} & \text{ in } B_1,\\[6pt]
[\hat{E}_{3, \rho} \times \nu] =0, [\hat{H}_{3, \rho} \times \nu] = -\hat{H}_{1, \rho}\times \nu |_{\ext}  &\text{ on } \partial B_1.
\end{cases}
\end{equation}
By Lemma \ref{lem:jump1},  we have 
\begin{equation}\label{equ:es}\|(\hat{E}_{3, \rho},\hat{H}_{3, \rho})\|_{H(\curl, B_5)} \leq C \Big(\|J_{\inte}\|_{L^2(B_1)} + 
 \|\hat{H}_{1, \rho}\times \nu|_{\ext} \|_{H^{-1/2}(\dive_{\Gamma}, \partial B_1)}\Big). 
 \end{equation}
Applying Lemma~\ref{lem:ex-12} to $(\hat E_{2, \rho}, \hat H_{2, \rho})$, by \eqref{hatvv}, we obtain
$$
  \|\hat{H}_{2, \rho}\times \nu|_{\ext} \|_{H^{-1/2}(\partial B_1)} + \|\hat{E}_{2, \rho}\cdot \nu|_{\ext} \|_{H^{-1/2}(\partial B_1)}  \le C \rho \| J_{\ext} \|_{L^2(\R^3\setminus B_2)}. 
$$
Since 
$$
(E_{2, \rho}, H_{2, \rho}) = (E, H) - (E_{1, \rho}, H_{1, \rho}) \mbox{ in } \mR^3 \setminus B_1, 
$$
it follows that
\begin{equation*}
  \|\hat{H}_{1, \rho}\times \nu|_{\ext} \|_{H^{-1/2}(\partial B_1)} + \|\hat{E}_{1, \rho}\cdot \nu|_{\ext} \|_{H^{-1/2}(\partial B_1)}  \le C \rho \| J_{\ext} \|_{L^2(\R^3\setminus B_2)}, 
 \end{equation*}
 which is \eqref{equ:es-1}. 
 
 \medskip
 Combining \eqref{equ:es-1} and \eqref{equ:es} yields
\begin{equation}\label{continue}
\| \big(\hat{E}_{3, \rho}, \hat{H}_{3, \rho} \big)\|_{H(\curl, B_5)} \leq C\left(\|J_{\inte}\|_{L^2(B_1)}+ \rho\|J_{\ext}\|_{L^2(\R^3\setminus B_2)}\right). \end{equation}
Applying Lemma \ref{lem-FF}, and using \eqref{continue}, we obtain
\begin{equation*}
\left|\left(\hat{E}_{3, \rho}\left(\frac{x}{\rho}\right), \hat{H}_{3, \rho}\left(\frac{x}{\rho}\right)\right)\right|\leq C\rho^3\left(\|J_{\inte}\|_{L^2(B_1)}+ \rho\|J_{\ext}\|_{L^2(\R^3\setminus B_2)}\right) \mbox{ for $x\in B_{3R}\setminus B_{2R}$}.
\end{equation*}
This implies (\ref{inequ2}). The proof is complete. \qed

\subsection{Approximate cloaking in the resonant case - Proofs of Theorems \ref{thm1.1} and \ref{thm1.2}} The key ingredient in the proof of Theorems~\ref{thm1.1} and \ref{thm1.2} is the following variant of Lemma~\ref{lem:jump1}. 

 \begin{lemma}
 	\label{lem:jump2}
 	Let $0 < \rho < \rho_0$, $\theta_{\rho } = (\theta_{1, \rho}, \theta_{2, \rho}) \in [L^2(D)]^6$, and $h_{\rho} = (h_{1, \rho}, h_{2, \rho}) \in [H^{-1/2}(\dive_{\Gamma}, \partial D)]^2$, and  let $(E_{\rho},H_{\rho})\in [\bigcap_{R > 1}H(\curl, B_R \setminus \partial D)]^2$   be the unique   radiating solution to the system
 	\[
 	\begin{cases}
 	\nabla\times E_{\rho} = i\rho H_{\rho} & \text{ in } \R^3\setminus D,\\[6pt]
 	\nabla\times H_{\rho} = -i\rho E_{\rho} & \text{ in } \R^3\setminus D,\\[6pt]
 	\nabla\times E_{\rho} = i\mu H_{\rho} + \theta_{1, \rho} & \text{ in } D,\\[6pt]
 	\nabla\times H_{\rho} = -i\eps E_{\rho} + \theta_{2, \rho} & \text{ in } D,\\[6pt]
 	[E_{\rho}\times \nu] = h_{1, \rho},  [H_{\rho}\times \nu] = h_{2, \rho} & \text{ on } \partial D.
 	\end{cases}
 	\]
Assume that $\mathcal N(D) \neq \{(0, 0)\}$. We have
 \begin{equation}
 \label{jump-eq-p2}
\|(E_{\rho}, H_{\rho})\|_{L^2(B_5)}\leq C\Big(\rho^{-1}\|\theta_{\rho}\|_{L^2(D)} + \|h_{\rho}\|_{H^{-1/2}(\partial D)} + \rho^{-1}\|\dive_{\Gamma}h_{\rho}\|_{H^{-1/2}(\partial D)}\Big).\end{equation}
Assume in addition that, for all $\rho \in (0, \rho_0)$,  
 \begin{equation} \label{cond}
 \int_{D} \big( \theta_{2, \rho} \cdot\bar{\bE} - \theta_{1, \rho} \cdot\bar{\bH} \big) \, dx = 0 \mbox{ for all } (\bE, \bH) \in \mathcal N(D). 
 \end{equation}
Then 
 	\begin{equation}\label{res-jump-eq-p1}\|(E_{\rho}, H_{\rho})\|_{L^2(B_5)}\leq C\Big(\|\theta_{\rho}\|_{L^2(D)} + \|h_{\rho}\|_{H^{-1/2}(\partial D)} + \rho^{-1}\|\dive_{\Gamma}h_{\rho}\|_{H^{-1/2}(\partial D)}\Big).\end{equation}
Here $C$ denotes a  positive constant depending only on $\rho_0, \,  \eps$,  and $\mu$.  Moreover, if 
 	\[ 
 	\lim_{\rho \to 0 } \Big( \|h_{\rho}\|_{H^{-1/2}(\partial D)} + \rho^{-1}\|\dive_{\Gamma}h_{\rho}\|_{H^{-1/2}(\partial D)} \Big) =  0 \quad \mbox{ and } \quad \lim_{\rho \to 0} \theta_\rho  =  \theta \mbox{ in } [L^2(D)]^6, 
	\]
	for some $\theta = (\theta_1, \theta_2) \in [L^2(D)]^6$, then 
 \begin{equation}\label{jump-convergence-p1}
 	\lim_{\rho \to 0} (E_{\rho}, H_{\rho}) =  Cl(\theta_1, \theta_2) \mbox{ in } [H(\curl, D)]^2.
 \end{equation}
\end{lemma}
 
\begin{remark} \rm In comparison with  \eqref{jump-eq} in Lemma~\ref{lem:jump1}, in the resonant case $\mathcal N(D) \neq  \{(0, 0)\}$, estimate \eqref{jump-eq-p2} is weaker. Under the compatibility condition~\eqref{cond}, estimate \eqref{res-jump-eq-p1} is stronger than 
\eqref{jump-eq-p2}. Note that the term $\| \dive_\Gamma h_\rho\|_{H^{-1/2}(\partial D)}$ in \eqref{jump-eq} of Lemma~\ref{lem:jump1} is replaced by $\rho^{-1} \| \dive_\Gamma h_\rho\|_{H^{-1/2}(\partial D)}$ in  \eqref{res-jump-eq-p1}. However, this does not affect the estimate for the degree of visibility in the compatible resonant case (in comparison with the non-resonant case) since in the proof of Theorem 1.2, we apply Lemma~\ref{lem:jump2} to the situation where  $\|h_{\rho}\|_{H^{-1/2}(\partial D)} $ and $\rho^{-1} \| \dive_\Gamma h_\rho\|_{H^{-1/2}(\partial D)}$ are of the same order. It is worth noting that the estimates in Lemma~\ref{lem:jump2} are somehow sharp because of the optimality of the estimates in Theorems~\ref{thm1.1} and \ref{thm1.2}; this is  discussed in Section~\ref{sect-opt}. 

\end{remark} 
 
\begin{proof} We will give the proof of \eqref{res-jump-eq-p1} and \eqref{jump-convergence-p1} and explain how to modify the proof of  \eqref{res-jump-eq-p1}  to obtain \eqref{jump-eq-p2}. 

\medskip 

We prove \eqref{res-jump-eq-p1} by contradiction. Assume that there exist  sequences $(\rho_n)_{n} \subset (0, \rho_0)$, $\big((E_n, H_n)\big)_{n} \subset [\bigcap_{R > 0}H(\curl, B_R \setminus \partial D)]^2$, $(\theta_{n})_{n} = \big((\theta_{1, n}, \theta_{2, n})\big)_{n}\subset [L^2(D)]^6$ such that \eqref{cond} holds for $(\theta_{1, n}, \theta_{2, n})$, 
 	\begin{equation}\label{equ:res-contra:main}
 	\begin{cases}
 	\nabla\times E_{n} = i\rho_{n} H_n & \text{ in } \R^3\setminus D,\\[6pt]
 	\nabla\times H_{n} = -i\rho_{n} E_n & \text{ in } \R^3\setminus D,\\[6pt]
 	\nabla\times E_{n} = i\mu H_n + \theta_{1, n} & \text{ in } D,\\[6pt]
 	\nabla\times H_{n} = -i\eps E_n + \theta_{2, n} & \text{ in } D,\\[6pt]
	[E_{n}\times \nu] = h_{1, n},  [H_{n}\times \nu] = h_{2, n} \text{ on } \partial D, 
 	\end{cases}
 	\end{equation}
 \begin{equation}\label{res-unit-jump}  
	 \|(E_n, H_n)\|_{L^2(B_5)} = 1 \text{ for all } n\in \mathbb{N},
\end{equation}
and
\begin{equation}\label{res-dive-jump12} 
	\lim_{n \to + \infty } \Big(  \|\theta_{n}\|_{L^2(D)} + \|h_{n}\|_{H^{-1/2}(\partial D)} + \rho_n^{-1}\|\dive_{\Gamma} h_n\|_{H^{-1/2}(\partial D)} \Big) = 0. 
\end{equation}
 	Without loss of generality, we assume that $\rho_n \rightarrow \rho_{\infty}\in [0, \rho_0]$. We will only consider the case $\rho_{\infty} = 0$. The proof in the case $\rho_{\infty} > 0$ follows similarly and is omitted.  
 	
 	Similar to \eqref{lem:jump-p1} and \eqref{lem:jump-p2}, we  have, by \eqref{res-dive-jump12},
	\begin{equation}\label{coucou1-jp}
	\lim_{n \to + \infty } \nabla \times E_n |_{\inte}  \cdot \nu  = 0  \quad  \mbox{ and } \quad \lim_{n \to + \infty } \nabla \times H_n  \cdot \nu   |_{\inte}  = 0 \mbox{ in } H^{-1/2}(\partial D). 
	\end{equation}
	Applying Lemma \ref{lem-compact-couple} and using \eqref{res-unit-jump}, without loss of generality, one might assume that $\big((E_n, H_n)\big)_n$ converges  in $[L^2(D)]^6$ 
	and hence also in $[L^2_{\loc}(\mR^3 \setminus  D)]^6$ by applying	 \eqref{lem:ex0-conclusion} of Lemma~\ref{lem:ex-12} and i) of Lemma~\ref{lem:compact} to $B_R \setminus D$.  Moreover, the limit $(E, H) \in [H_{\loc}(\curl, \R^3)]^2$ satisfies 
 	\begin{equation}\label{equ:res-contra:limit}
 	\begin{cases}
 	\nabla\times E = \nabla\times H = 0  & \text{ in } \R^3\setminus D,\\[6pt]
 	\dive E = \dive H = 0 & \text{ in }  \R^3\setminus D,\\[6pt]
 	\nabla\times E = i\mu H  & \text{ in } D,\\[6pt]
 	\nabla\times H = -i\eps E  & \text{ in } D,
 	\end{cases}
 	\end{equation}
 	and, by applying Lemma~\ref{lem-SC-1} and letting $\rho_n \to 0$, 
 	\begin{equation}\label{res-decay-jump}
 	|\big(E(x), H(x)\big)| = O(|x|^{-2}) \mbox{ for large } |x|. 
 	\end{equation} 
Since
 	\begin{equation*}\int_{D} \big( \theta_{2, n}\cdot\bar{\bE} - \theta_{1, n}\cdot\bar{\bH} \big) \,dx = 0  \mbox{ for all } (\bE, \bH) \in \mathcal N(D), \end{equation*}
 	by Lemma \ref{lem:fredholm},  there exists a unique $(E_{1, n}, H_{1, n}) \in \mathcal N(D)^{\perp}$ solving
 	\[
 	\begin{cases}
 	\nabla\times E_{1, n} = i\mu H_{1, n} + \theta_{1, n} & \text{ in } D,\\[6pt]
 	\nabla\times H_{1, n} = -i\eps E_{1, n} + \theta_{2, n} & \text{ in } D,\\[6pt]
 	\nabla\times E_{1, n}\cdot \nu = \nabla\times H_{1, n}\cdot \nu = 0 & \text{ on } \partial D.
 	\end{cases}
 	\]
 	Denote by $(E_{2, n}, H_{2, n})$ the projection 
 	of $(E_n, H_n) - (E_{1, n}, H_{1, n})$ onto $\mathcal N(D)$ and define 
 	\[\tilde{E}_n = \rho^{-1}_{n}(E_n - E_{1, n} - E_{2, n})  \quad \mbox{ and } \quad \tilde{H}_n = \rho^{-1}_{n}(H_n - H_{1, n} - H_{2, n}) \quad \mbox{ in } D.\]
 	Then 
 	\begin{equation}\label{lem:jump-perp}
 	(\tilde{E}_n, \tilde{H}_n )\in \mathcal N(D)^{\perp}
 	\end{equation}
 	and 
 	\begin{equation}\label{equ:contra:ortho}
 	\begin{cases}
 	\nabla\times \tilde{E}_n = i\mu \tilde{H}_n &\text{ in } D,\\[6pt]
 	\nabla\times \tilde{H}_n = -i\eps \tilde{E}_n &\text{ in } D,\\[6pt]
 	\nabla\times \tilde{E}_n \cdot \nu = \rho^{-1}_n\nabla\times E_n\cdot \nu |_{\inte} &\text{ on } \partial D,\\[6pt]
 	\nabla\times \tilde{H}_n\cdot \nu = \rho^{-1}_n\nabla\times H_n\cdot \nu |_{\inte} &\text{ on } \partial D.\\
 	\end{cases}
 	\end{equation}
 	We have  
 	\begin{equation*} \rho^{-1}_n\nabla\times E_n\cdot \nu |_{\inte} = \rho^{-1}_n  \nabla\times E_n\cdot \nu|_{\ext} + \rho_n^{-1}\dive_{\Gamma}h_{1, n} =  iH_n\cdot \nu|_{\ext} + \rho_n^{-1}\dive_{\Gamma}h_{1, n}\mbox{ on } \partial D.
 	\end{equation*}
 	This implies, by \eqref{equ:contra:ortho},  
 	\begin{equation}\label{lem:jump-t1}
 	\mu\tilde{H}_n\cdot\nu =   H_n\cdot \nu|_{\ext} - i\rho_n^{-1}\dive_{\Gamma}h_{1, n}\text{ on } \partial D.
 	\end{equation}
Similarly, we have
 	\begin{equation}\label{lem:jump-t2}
 	\eps \tilde{E}_n \cdot\nu =   E_n\cdot \nu|_{\ext} - i\rho_n^{-1}\dive_{\Gamma}h_{2, n}\text{ on } \partial D.
 	\end{equation}
Using \eqref{res-dive-jump12}, we derive from \eqref{coucou1-jp}, \eqref{lem:jump-t1}, and \eqref{lem:jump-t2} that 
 	\begin{equation}
 	\label{boun-jump}
 	(\eps\tilde{E}_n\cdot\nu, \mu\tilde{H}_n\cdot \nu)  \rightarrow (E\cdot\nu|_{\ext}, H\cdot \nu|_{\ext}) \mbox{ in } H^{-1/2}(\partial D) \mbox{ as $n\rightarrow \infty$.}
 	\end{equation}
 	It follows from Lemma \ref{stabilityinterior} below that $\Big((\tilde E_n, \tilde H_n) \Big)_{n}$ is bounded in $[L^2(D)]^6$. Applying Lemma \ref{lem-compact-couple} to $(\tilde E_n, \tilde H_n)$, one can assume that
 	 \begin{equation}\label{res-convergence-aux}(\tilde{E}_n, \tilde{H}_n) \mbox{ converges to some } (\tilde{E}, \tilde{H})\in \mathcal N(D)^{\perp} \mbox{ in }  [H(\curl, D)]^2.\end{equation}
 	 Moreover,  from \eqref{equ:contra:ortho} and \eqref{boun-jump}, we have 
 	\begin{equation}\label{equ:contra:limit2}
 	\begin{cases}
 	\nabla\times \tilde{E} = i\mu \tilde{H} &\text{ in } D,\\[6pt]
 	\nabla\times \tilde{H} = -i\eps \tilde{E} &\text{ in } D,\\[6pt]
 	\eps\tilde{E}\cdot \nu = E\cdot \nu |_{\ext} &\text{ on } \partial D,\\[6pt]
 	\mu\tilde{H}\cdot \nu = H\cdot \nu |_{\ext} &\text{ on } \partial D.\\
 	\end{cases}
 	\end{equation}
 	Applying Lemma \ref{lem:uniqueness} to $(E, H)$ defined in $\mR^3$ and $(\tilde E, \tilde H)$ defined in $D$ and 
 	using \eqref{equ:res-contra:limit}, \eqref{res-decay-jump}, and \eqref{equ:contra:limit2}, we deduce   that $E= H = 0$ in $\R^3$, which contradicts  \eqref{res-unit-jump}. The proof of \eqref{res-jump-eq-p1} is complete.
 	
 	\medskip 
 	We next establish \eqref{jump-convergence-p1}.  Fix a sequence $(\rho_n)$ converging to $0$. From \eqref{res-jump-eq-p1}, one obtains that 
 	\[\|\big(E_{\rho_{n}}, H_{\rho_n}\big)\|_{L^2(B_5)}\leq C\Big(\|\theta_{\rho_n}\|_{L^2(D)}+\|h_{\rho}\|_{H^{-1/2}(\partial D)} + \rho_n^{-1}\| \dive_{\Gamma}h_{\rho_n}\|_{H^{-1/2}(\partial D)}\Big)\leq C.\]
 	Define $(\tilde E_{\rho_n}, \tilde H_{\rho_n})$ in $D$ from $(E_{\rho_n}, H_{\rho_n})$  as in the definition of $(\tilde E_{n}, \tilde H_{n})$ from $(E_{n}, H_{n})$.
 	Using the same arguments to obtain \eqref{res-convergence-aux}, we have 
 	\begin{equation}\label{unit-jump-4-1}
 	(\tilde{E}_{\rho_n}, \tilde{H}_{\rho_n}) \mbox{ converges to } (\tilde{E}, \tilde{H})\in \mathcal N(D)^{\perp} \mbox{ in }  [H(\curl, D)]^2.\end{equation}
 	Up to a subsequence,  $(E_{\rho_n}, H_{\rho_n})$ converges to $(E, H)$ in  $\big[H_{\loc}(\curl, \mR^3) \big]^2$ and 
 	\begin{equation}\label{decay-cond}
 	|\big(E(x), H(x)\big)| = O(|x|^{-2}) \mbox{ for large } |x|. 
 	\end{equation}
 	Moreover, as in \eqref{equ:contra:limit2}, one can show that \eqref{defcase2.1} holds. Since the limit is unique,  assertion~ \eqref{jump-convergence-p1} follows. 

\medskip 

We finally show how to modify the proof of \eqref{res-jump-eq-p1} to obtain \eqref{jump-eq-p2}.  The proof is also based on a  contradiction argument and is similar to the one of \eqref{res-jump-eq-p1}. However,  we denote by $(E_{2, n}, H_{2, n})$ the projection of $(E_n, H_n) $ onto $\mathcal N$ (note that $E_{1, n}$ and $H_{1, n}$ might not exist in this case)) and define 
 	\[\tilde{E}_n = \rho^{-1}_{n}(E_n - E_{2, n}) \mbox{ in } D \quad \mbox{ and } \quad \tilde{H}_n = \rho^{-1}_{n}(H_n - H_{2, n}) \mbox{ in } D.\]
 	Then 
 	\begin{equation}
 	\begin{cases}
 	\nabla\times \tilde{E}_n = i\mu \tilde{H}_n + \rho_n^{-1}\theta_{1, n} &\text{ in } D,\\[6pt]
 	\nabla\times \tilde{H}_n = -i\eps \tilde{E}_n + \rho_n^{-1}\theta_{2, n} &\text{ in } D,\\[6pt]
 	\nabla\times \tilde{E}_n \cdot \nu = \rho^{-1}_n\nabla\times E_n\cdot \nu |_{\inte} &\text{ on } \partial D,\\[6pt]
 	\nabla\times \tilde{H}_n\cdot \nu = \rho^{-1}_n\nabla\times H_n\cdot \nu |_{\inte} &\text{ on } \partial D.\\
 	\end{cases}
 	\end{equation}
 	 Since $(\rho_n^{-1} \theta_n)_n \to (0, 0)$ in $[L^2(D)]^6$, the sequence $\big((\tilde E_n, \tilde H_n) \big)_n$ converges to $(\tilde E, \tilde H)$ in $[L^2(D)]^6$. Similar to the proof of \eqref{res-jump-eq-p1}, one also derives that $(E, H) = (0, 0) \mbox{ in } \R^3$. This yields  a contradiction. The proof is complete. 
 \end{proof}

In the proof of Lemma~\ref{lem:jump2}, we used the following lemma: 

\begin{lemma}\label{stabilityinterior}
Assume that $D$ is simply connected and $(E, H)\in \mathcal N(D)^{\perp}$ satisfies
\begin{equation}
\nabla\times E = i\mu H \mbox{ in } D \quad \mbox{ and } \quad 
\nabla\times H = -i\eps E \mbox{ in } D. 
\end{equation}
We have 
\[\|(E, H)\|_{H(\curl, D)}\leq C \|(\eps E\cdot \nu, \, \mu H\cdot\nu)\|_{H^{-1/2}(\partial D)},\]
for some positive constant $C$ depending only on $D$, $\eps, \mu$. 
\end{lemma}
\begin{proof}
It suffices to prove that
\begin{equation}
\|(E, H)\|_{L^2(D)}\leq C \|(\eps E\cdot \nu, \mu H\cdot\nu)\|_{H^{-1/2}(\partial D)}.
\end{equation}
The proof is via a standard contradiction argument. Assume that there exists a sequence
$\big((E_n, H_n)\big)_{n}\subset \mathcal N(D)^{\perp}$ such that 
\begin{equation}
\nabla\times E_n = i\mu H_n \mbox{ in } D \quad  \mbox{ and } \quad 
\nabla\times H_n = -i\eps E_n \mbox{ in } D,
\end{equation}
\begin{equation}\label{estimate-unit}
\|(E_n, H_n)\|_{L^2(D)} = 1 \mbox{ for all $n$,}
\end{equation}	
and 
\begin{equation}
\big(\eps E_n\cdot \nu, \, \mu H_n\cdot \nu \big)\to 0 \mbox{ in $ [H^{-1/2}(\partial D)]^2$.}
\end{equation}
Applying Lemma \ref{lem-compact-couple}, one might assume that $(E_n, H_n)$ converges to some  $(E_0, H_0)$  in $[H(\curl, D)]^2$. Then $(E_0, H_0) \in {\mathcal N}(D)^{\perp}$ and
\begin{equation}
\begin{cases}
\nabla\times E_0 = i\mu H_0 &\mbox{ in } D, \\[6pt]
\nabla\times H_0 = -i\eps E_0 &\mbox{ in } D,\\[6pt]
\nabla\times E_0\cdot \nu = \nabla\times H_0\cdot \nu = 0 &\text{ on } \partial D. 
\end{cases}
\end{equation}
It follows that  $(E_0, H_0)\in {\mathcal N(D)}^{\perp} \cap \mathcal N(D)$. Hence $(E_0, H_0) = (0, 0) \mbox{ in } D$, which contradicts \eqref{estimate-unit}.
\end{proof}

\medskip 
We are ready to give the 

\medskip

\noindent{\bf Proof of Theorem \ref{thm1.1}.} In this proof, we use the same notations as in the one of Theorem~\ref{thm1}. 
Similar to the proof of Theorem~\ref{thm1}, using Lemmas \ref{lem:ex-12} and \ref{lem-FF}, we have, for $R > 1/2$,  
\begin{equation}
\label{inequ1-thm1.1}
\|(E_{2, \rho}, H_{2, \rho})\|_{L^2(B_{3R}\setminus B_{2R})}\leq C\rho^3\|J_{\ext}\|_{L^2(\R^3\setminus B_2)}.
\end{equation}
Involving the same method used to prove \eqref{equ:es-1} and \eqref{inequ2}, however, applying \eqref{res-jump-eq-p1} in Lemma~\ref{lem:jump2} instead of Lemma~\ref{lem:jump1}, we have 
\begin{equation}\label{equ:es-1-thm1.1}
  \|\hat{H}_{1, \rho}\times \nu|_{\ext} \|_{H^{-1/2}(\partial B_1)} + \|\hat{E}_{1, \rho}\cdot \nu|_{\ext} \|_{H^{-1/2}(\partial B_1)}  \le C \rho \| J_{\ext} \|_{L^2(\R^3\setminus B_2)}
 \end{equation}
and
\begin{equation}
\label{inequ2-thm1.1}
\|(E_{3, \rho}, H_{3, \rho})\|_{L^2(B_{3R}\setminus B_{2R})}\leq C\big(\rho^3\|J_{\ext}\|_{L^2(\R^3\setminus B_2)} + \rho^2 \|J_{\inte}\|_{L^2(B_1)}\big).
\end{equation}
It is clear that  \eqref{thm1.1-est1}  follows from  (\ref{inequ1-thm1.1}) and  (\ref{inequ2-thm1.1}). Moreover, by Lemma~\ref{lem:jump1}, assertion  \eqref{thm1.1-CV} now follows from \eqref{equ:es-1-thm1.1}  and the fact that $(E_c, H_c) = (\hat E_{3, \rho}, \hat H_{3, \rho})$ in $B_1$.  \qed

\bigskip 
\noindent{\bf Proof of Theorem \ref{thm1.2}.} In this proof, we use the same notations as in the one of Theorem~\ref{thm1}. 
Similar to the proof of Theorem~\ref{thm1}, using Lemmas \ref{lem:ex-12} and \ref{lem-FF}, we have, for $R > 1/2$,  
\begin{equation}
\label{inequ1-thm1.2}
\|(E_{2, \rho}, H_{2, \rho})\|_{L^2(B_{3R}\setminus B_{2R})}\leq C\rho^3\|J_{\ext}\|_{L^2(\R^3\setminus B_2)}.
\end{equation}
Using the same method used to prove \eqref{inequ2}, however, applying \eqref{jump-eq-p2} in Lemma~\ref{lem:jump2} instead of Lemma~\ref{lem:jump1}, we have 
\begin{equation}\label{equ:es-1-thm1.2}
  \|\hat{H}_{1, \rho}\times \nu|_{\ext} \|_{H^{-1/2}(\partial B_1)} + \|\hat{E}_{1, \rho}\cdot \nu|_{\ext} \|_{H^{-1/2}(\partial B_1)}  \le C \rho \| J_{\ext} \|_{L^2(\R^3\setminus B_2)}
 \end{equation}
and
\begin{equation}
\label{inequ2-thm1.2}
\|(E_{3, \rho}, H_{3, \rho})\|_{L^2(B_{3R}\setminus B_{2R})}\leq C\big( \rho^3\|J_{\ext}\|_{L^2(\R^3\setminus B_2)} + \rho \|J_{\inte}\|_{L^2(B_1)}\big).
\end{equation}
It is clear that  \eqref{thm1.2-est1}  follows from  (\ref{inequ1-thm1.2}) and  (\ref{inequ2-thm1.2}). 

\medskip 

It remains to prove   \eqref{explosion}.  Using the linearity of the system and applying Theorem~\ref{thm1.1}, one can assume that $J_{\ext} = 0$, and 
   $J_{\inte} = \bE_0 $ for some $(\bE_0, \bH_0)\in \mathcal N \setminus \{(0, 0)\}$. From the definition of ${\mathcal N}$, we have 
   $$
   \bE_0 \not \equiv 0 \quad \mbox{ and } \quad \bH_0 \not \equiv 0 \mbox{ in } B_1.  
   $$  
Note that $(\hat{E}_c, \hat{H}_c)\in [H_{\loc}(\curl, \R^3)]^2$  is the unique  radiating solution to the system
   \begin{equation}\label{eqEc}
   \begin{cases}
   \nabla \times \hat{E}_c = i\omega \rho \hat{H}_c  & \text{ in } \R^3\setminus B_1,\\[6pt]
   \nabla \times \hat{H}_c = -i\omega \rho \hat{E}_c  & \text{ in } \R^3\setminus B_1,\\[6pt]
   \nabla \times \hat{E}_c = i\omega \mu \hat{H}_c  & \text{ in }  B_1,\\[6pt]
   \nabla \times \hat{H}_c = -i\omega \eps \hat{E}_c  + \bE_0 & \text{ in }  B_1.
   \end{cases}
   \end{equation}
   We prove \eqref{explosion} by contradiction. Assume that there exists a sequence $\big(\rho_n\big)_{n }\subset (0, 1/2)$ converging to $0$ such that 
   \begin{equation}
   \label{explosion:contr}
   \lim\limits_{n\rightarrow \infty} \rho_n\|(E_n, H_n)\|_{L^2(B_1)} = 0,
   \end{equation}
   where $(E_n, H_n) \in [H_{\loc}(\curl, \R^3)]^2$ is the unique  radiating solution to the system
   \begin{equation}\label{equ:contra:case2}
   \begin{cases}
   \nabla \times E_n = i\omega \rho_n H_n  & \text{ in } \R^3\setminus B_1,\\[6pt]
   \nabla \times H_n = -i\omega \rho_n E_n  & \text{ in } \R^3\setminus B_1,\\[6pt]
   \nabla \times E_n = i\omega \mu H_n  & \text{ in }  B_1,\\[6pt]
   \nabla \times H_n = -i\omega \eps E_n  + \bE_0 & \text{ in }  B_1.
   \end{cases}
   \end{equation}
   Applying Lemma \ref{lem:trace} to $(E_n, H_n)$ in $B_1$ and  using (\ref{explosion:contr}) and \eqref{equ:contra:case2}, we obtain
   \begin{equation}
   \lim \limits_{n\rightarrow \infty} \rho_n \| \big(E_n\times \nu, H_n\times \nu \big) \|_{H^{-1/2}( \partial B_1)} = 0.
   \end{equation}
   By Lemma \ref{lem:ex-12}, we have 
   \begin{equation}\label{equ:explo:1} \lim \limits_{n\rightarrow \infty} \rho_n \| \big(E_n, H_n\big)\|_{L^2( B_2\setminus B_1)} = 0.\end{equation}
   Since $\dive{E_n} = \dive{H_n} = 0$ in $\R^3\setminus B_1$, we have,  by Lemma \ref{lem:trace} and \eqref{equ:explo:1}, 
   \[\lim \limits_{n\rightarrow \infty} \rho_n \| \big(E_n \cdot \nu,   H_n\cdot \nu \big) \|_{H^{-1/2}(\partial B_1)} = 0.\]
  It follows that 
   \begin{equation}
   \label{con:blow}
   \lim \limits_{n\rightarrow \infty} \|\big( \dive_\Gamma (E_n \times \nu ), \dive_\Gamma (H_n\times \nu ) \big)\|_{H^{-1/2}(\partial B_1)} = \lim \limits_{n\rightarrow \infty} \|\big( \nabla \times E_n \cdot \nu, \nabla \times H_n \cdot  \nu ) \big)\|_{H^{-1/2}(\partial B_1)} = 0.
   \end{equation}
Using  the fact that $(\bE_0, \bH_0)\in \mathcal N$, we derive from  \eqref{equ:contra:case2} that 
   \[\int_{B_1}\mu^{-1}\nabla \times \bar{\bE}_0\cdot\nabla \times E_n\, dx - \omega^2\int_{B_1}\eps \bar{\bE}_0\cdot E_n\, dx = -i\omega\int_{\partial B_1}(\nu\times E_n)\cdot\bar{\bH}_0 ds,\]
   and
   \[\int_{B_1}\mu^{-1}\nabla \times E_n \cdot\nabla \times \bar{\bE}_0\,dx - \omega^2\int_{B_1}\eps E_n\cdot \bar{\bE}_0\,dx = -i\omega\int_{\partial B_1}(\nu\times H_n)\cdot \bar{\bE}_0\, ds + i\omega\int_{B_1}\bE_0\cdot\bar{\bE}_0.\]
Considering the imaginary part of the two identities yields 
   \begin{equation}\label{con:equa} \Re\left\{\int_{\partial B_1}(\nu\times H_n)\cdot \bar{\bE}_0  ds + \int_{\partial B_1}(\nu\times E_n)\cdot \bar{\bH}_0 ds\right\}  = \int_{B_1}|\bE_0|^2 > 0 .\end{equation}
   However, since $ \nabla\times \bH_0 \cdot \nu = 0$ on $\partial B_1$, by Lemma \ref{potential}, there exists $\bH \in H(\curl, B_1)$ such that \[\nabla\times \bH_0 = \nabla\times \bH \mbox{ in $B_1$} \quad \mbox{ and } \quad \bH\times \nu = 0 \mbox{ on } \partial B_1. \] Since $\nabla\times (\bH_0 - \bH) = 0$ in $B_1$, by Lemma \ref{lem:potential}, there exists $\xi\in H^1(B_1)$ such that
   \[\bH_0 - \bH = \nabla \xi \mbox{ in } B_1,\]
   and hence 
   \[\bH_0\times \nu = \nabla\xi \times \nu \mbox{ on } \partial B_1.\]
   We have thus
   \begin{equation}\label{thm1.2-toto1}
   \int_{\partial B_1}(\nu \times E_n)\cdot \bar{\bH}_0 \, ds = \int_{\partial B_1}(\nu\times E_n)\cdot \nabla \bar{\xi}  \, ds = \int_{\partial B_1}\dive_{\Gamma}(\nu\times E_n)\, \bar{\xi}  \, ds \to 0 \mbox{ as } n \to + \infty, 
   \end{equation}
   thanks to (\ref{con:blow}). Similarly,  we obtain 
    \begin{equation}\label{thm1.2-toto2}
   \int_{\partial B_1}(\nu\times H_n)\cdot \bar{\bE}_0 \, ds \to 0 \mbox{ as } n \to + \infty. 
   \end{equation}
   Combining  (\ref{con:equa}), \eqref{thm1.2-toto1}, and \eqref{thm1.2-toto2}, we obtain a contradiction. Hence (\ref{explosion}) holds. The proof is complete.\qed

   \section{Optimality of the degree of visibility}\label{sect-opt}
   
   In this section, we present various settings that justify the optimality of the degree of visibility in  Theorems \ref{thm1}, \ref{thm1.1}, and \ref{thm1.2}. In what follows in this section, we {\bf assume} that 
   \begin{equation}\label{em-O}
   \eps = \mu = I \mbox{ (the identity matrix)    in }  B_1.
   \end{equation}   
Let  $h_n^{(1)}$ $(n \in \N)$ be the spherical Hankel function of first kind of order $n$ and let $j_n, \ y_n$ denote respectively its real and imaginary parts.  
For $-n\leq m\leq n, n\in \N$, denote $Y_n^m$ the spherical harmonic function of order $n$ and degree $m$ and set    
$$
U_n^m(\hat x) := \nabla_{\partial B_1}Y_n^m(\hat x) \quad \mbox{ and } \quad V_n^m(\hat x) := \hat x \times U_n^m(\hat x) \mbox{ for } \hat x \in \partial B_1.
$$
We recall that $Y_n^m(\hat x)\hat x$, $U_n^m(\hat x)$, and $V_n^m(\hat x)$ for $-n\leq m\leq n, n\in \N$ form an orthonormal basis of $[L^2(\partial B_1)]^3$.   


\medskip 
We have 
   
\begin{lemma}\label{opt-char} System \eqref{medium:cloak} is non-resonant if and only if $j_n(\omega)\neq 0$ for all $n \ge 1$.
\end{lemma}

\begin{proof} Assume that $j_n(\omega) = 0$ for some $n \ge 1$. Fix such an $n$ and define, in $B_1$, 
\begin{equation*}
\bE_0(x) = j_n(\omega r)V_n^0(\hat x)  \mbox{  and  } \bH_0(x) = \frac{n(n+1)}{i\omega r}j_n(\omega r)Y_n^0(\hat x)\hat x + \frac{1}{i\omega r}[j_n(\omega r) + \omega r j'_n(\omega r)]U^0_n(\hat x),  
\end{equation*}
where $r = |x|$ and $\hat x = x / |x|$.  Then $(\bE_0, \bH_0) \in \mathcal N$. System \eqref{medium:cloak} is hence resonant.  Conversely, assume that $j_n(\omega)\neq 0$ for all $n\in \N$. Using separation of variables (see, e.g., \cite[Theorem 2.48]{Kirsch}), one can check that if $(\bE_0, \bH_0) \in \mathcal N$ then $(\bE_0, \bH_0) = (0, 0)$ in $B_1$. 
\end{proof}
   
The following result implies  the optimality of \eqref{thm1-est1} with respect to $J_{\ext}$.
For computational ease, instead of considering fields generated by $J_{\ext}$, we deal with fields generated by a plane wave. 
In what follows, we assume that $0 < \rho < 1/2$.  We have

\begin{proposition}\label{prop-nonresonant1} Set  $v(x) : = (0, 1, 0)e^{i\omega x_3}$ for $ x \in \R^3$. 
There exists $\omega > 0$  such that
$$
\|E_c \|_{L^2(B_4\setminus B_2)} \ge C \rho^3, 
$$
for some positive constant $C$ independent of $\rho$. Here $(E_c, H_c) \in [H_{\loc}(\curl, \R^3)]^2$  is uniquely determined by 
\begin{equation*}
\begin{cases}
\nabla \times E = i\omega \mu_c H &\text{ in } \mathbb{R}^3,\\[6pt]
\nabla \times H = -i\omega \eps_c E&\text{ in } \mathbb{R}^3, 
\end{cases}
\end{equation*}
where $E  = E_c  + v $ and $H = H_c + \frac{1}{i \omega} \nabla\times v$ and by the radiation condition.  Here $(\eps_c, \mu_c)$ is defined by  \eqref{medium:cloak} where $(\eps, \mu)$ is given in \eqref{em-O}. 
\end{proposition}

 \begin{proof} Let $\omega > 0$ be  such that $j_1(\omega) \neq 0$. Set 
$$
 (\pbE_{\rho}, \pbH_{\rho}) = (F^{-1}_{\rho}*E, F^{-1}_{\rho}*H) \mbox{ in } \R^3,
$$
and define 
   	\[
   	(\bE_{\rho},  \bH_{\rho}) = \left\{\begin{array}{cl}
   	 (\pbE_{\rho} - v, \pbH_{\rho} - \frac{1}{i \omega} \nabla\times v ) &\mbox{ in } \R^3\setminus B_{\rho},\\[6pt]
   	 (\pbE_{\rho}, \pbH_{\rho}) &\mbox{ in } B_{\rho}.
   	\end{array} \right.
   	\]
	Set 
	$$
	(\tilde  \bE_{\rho},  \tilde \bH_{\rho})  =(\bE_{\rho},  \bH_{\rho}) (\rho \, \cdot) \quad \mbox{ and } \quad \tilde v =  v (\rho \,  \cdot) \quad \mbox{ in } \R^3. 
	$$
   	We have 
   	\begin{equation}\label{eq-optimal2}
   	\begin{cases}
   	\nabla\times \tilde \bE_{\rho} = i\rho\omega \tilde \bH_{\rho}  &\mbox{ in } \R^3\setminus B_1,\\[6pt]
   	\nabla\times \tilde \bH_{\rho} = -i\rho\omega \tilde \bE_{\rho}  &\mbox{ in } \R^3\setminus B_1,\\[6pt]
   	\nabla\times \tilde \bE_{\rho} = i\omega \tilde \bH_{\rho}  &\mbox{ in } B_1,\\[6pt]
   	\nabla\times \tilde \bH_{\rho} = -i\omega \tilde \bE_{\rho}  &\mbox{ in } B_1,\\[6pt]
   	[\tilde \bE_{\rho}\times \nu] = - \tilde v \times \nu, \quad  [\tilde \bH_{\rho}\times \nu] = -  \frac{1}{i \rho \omega} (\nabla\times \tilde v)  \times \nu &\mbox{ on } \partial B_{1}. 
   	\end{cases}
   	\end{equation}
 Denote 
   	\[
   	A_{\ext} = \int_{\partial B_1}\tilde \bE_{\rho}|_{\ext}\cdot \bar V_1^1 \,ds \quad \mbox{ and } \quad  A_{\inte} = \int_{\partial B_1}\tilde  \bE_{\rho}|_{\inte}\cdot \bar V_1^1\,ds. 
   	\]
Using the transmission condition for $\tilde \bE_{\rho}\times\nu$ on $\partial B_1$ and  considering only the component with respect to $V_1^1$ for $\tilde \bE_{\rho}$ (see, e.g., \cite[Theorem 2.48]{Kirsch}), we have 
   	\begin{equation}
   	\label{order3-eq2.1}
   	 A_{\ext} -  A_{\inte} = \alpha,
   	 \end{equation}
   	where 
   	\[
   	\alpha = -  \int_{\partial B_1} \tilde v \cdot \bar V_{1}^1 \  ds.
	   	\]
 Using the transmission condition for $\tilde \bH_{\rho}\times\nu$ on $\partial B_1$ and considering the component with respect to $U_1^1$ for $\tilde \bH_{\rho}$ (see, e.g., \cite[Theorem 2.48]{Kirsch}), we have
   	\begin{equation}\label{order3-eq2.2}
   	\dsp a_{\ext}(\omega \rho) A_{\ext} - a_{\inte}(\omega) A_{\inte} = \beta,
   	\end{equation}
   	where
	\[
	a_{\ext}(r) = \frac{\big(h^{(1)}_1(r)+ rh'^{(1)}_1(r) \big)}{- i r h^{(1)}_1(r)}, \quad  \quad a_{\inte}(r) = \frac{\big(j_1(r)+ rj'_1(r) \big)}{-i r j_1(r)},
  \quad \mbox{ and } \quad \beta = \alpha a_{\inte}(\omega \rho).\] 
Combining \eqref{order3-eq2.1} and \eqref{order3-eq2.2} yields
\begin{equation}\label{pro1-O-part0}
A_{\ext} = \frac{\beta- \alpha a_{\inte} (\omega)}{a_{\ext}(\omega \rho) - a_{\inte}(\omega)}. 
\end{equation}
Since
\begin{equation}\label{pro1-O-part0.1}
   h^{(1)}_1(x) = i\frac{d}{d x}\frac{e^{ix}}{x} = \frac{\sin x - x\cos x}{x^2} +i \frac{x\sin x - \cos x}{x^2}, \mbox{ for } x\in \R,
   \end{equation}
we derive that 
\begin{equation}\label{pro1-O-part1}
   	\liminf\limits_{\rho \to 0}\rho^{-1}\left| a_{\ext} (\omega \rho) - a_{\inte} (\omega) \right|^{-1}  > 0.
\end{equation}
Since, by separation of variables, (see, e.g., \cite[Theorem 2.48]{Kirsch}), 
\[
\left|\int_{\partial B_1}\tilde v \cdot \bar V_1^1\, ds \right| = \left| \frac{j_1(\omega \rho)}{j_1(\omega)}\int_{\partial B_1} v \cdot  \bar V_1^1\, ds\right|,
\] 
we have
\begin{equation}\label{pro1-O-part1.1}
C^{-1}\rho \leq|\alpha| \leq C\rho
\end{equation}
for some positive constant $C$ independent of $\rho$. From \eqref{pro1-O-part1.1} and the fact that 
\[
|a_{\inte}(\omega \rho)| \geq C\rho^{-1},
\]
we have  
   	\begin{equation}\label{pro1-O-part2}
   	\liminf\limits_{\rho \to 0}\left| \beta -  \alpha a_{\inte}(\omega) \right|> 0. 
   	\end{equation}
Combining \eqref{pro1-O-part1} and \eqref{pro1-O-part2} yields
\begin{equation}\label{pro1-O-part3}
\liminf_{\rho \to 0} \rho^{-1}|A_{\ext}| > 0. 
\end{equation}
Since, again by separation of variables,
   	\begin{equation*}
   	\int_{\partial B_1} \tilde\bE_\rho (r \hat x) \cdot  \bar V^{1}_1 (\hat x) \, d \hat x =  \frac{h^{(1)}_1(\omega\rho r)}{h^{(1)}_1(\omega\rho )}  A_{\ext}, 
   	\end{equation*}
and, by Lemma~\ref{pre:cha},  
$$
\tilde \bE_\rho (x/ \rho) = \bE_\rho (x) = \mathcal E_\rho(x) - v(x) = E_c(x)  \mbox{ for } x \in B_4 \setminus B_2, 
$$
we obtain the conclusion  from \eqref{pro1-O-part0.1} and  \eqref{pro1-O-part3}.  \end{proof}

We next show the optimality of  \eqref{thm1-est1} with respect to $J_{\inte}$.
   
   \begin{proposition} \label{pro-O-2} Assume that the system is non-resonant and  $J_{\ext} = 0$ in $\R^3\setminus B_2$. There exists $J_{\inte}\in [L^2(B_1)]^3$ such that 
   	\[ \liminf_{\rho \to 0 } \rho^{-2 } \left\| H_c \right\|_{L^2(B_4\setminus B_2)} > 0.\]
   
   \end{proposition}

   \begin{proof}  Consider 
\begin{equation}\label{def-Jint}
   	J_{\inte}(x) =  j_1(\omega r) V_1^1(\hat x)  \mbox{ in $B_1$},  
\end{equation}
where $r = |x|$ and $\hat x = x / |x|$. 
Set
   	\[\bE_0 = J_{\inte} \mbox{ and } \bH_0 = \frac{1}{i \omega} \nabla\times  \bE_0  \quad \mbox{ in } B_1.\]
   	Then 
   	\begin{equation}\label{proo-resonant-p2}\begin{cases}
   	\nabla\times \bE_0 = i\omega \bH_0 &\mbox{ in } B_1,\\[6pt]
   	\nabla\times \bH_0  = -i\omega \bE_0 & \mbox{ in } B_1.   
	\end{cases}\end{equation}
	Define   
	\[ (\hat  \bE_{\rho},  \hat \bH_{\rho})  = \rho (\mathcal E_{\rho},  \mathcal H_{\rho}) (\rho \cdot) \mbox{ in } \R^3, 
\]
where $(\mathcal E_{\rho},  \mathcal H_{\rho})$ is given in \eqref{def-EHrho}.
Then 
\begin{equation*}
   	\begin{cases}
   	\nabla\times \hat \bE_{\rho} = i\rho\omega \hat \bH_{\rho}  &\mbox{ in } \R^3\setminus B_1,\\[6pt]
   	\nabla\times \hat \bH_{\rho} = -i\rho\omega \hat \bE_{\rho}  &\mbox{ in } \R^3\setminus B_1,\\[6pt]
   	\nabla\times \hat \bE_{\rho} = i\omega \hat \bH_{\rho}  &\mbox{ in } B_1,\\[6pt]
   	\nabla\times \hat \bH_{\rho} = -i\omega \hat \bE_{\rho} + \bE_0 &\mbox{ in } B_1. 
   	\end{cases}
   	\end{equation*}
   We have  
   	\begin{equation}\label{con:equa1-1} \int_{\partial B_1}(\nu\times \hat \bH_{\rho})\cdot \bar{\bE}_0  ds - \int_{\partial B_1}(\nu\times \hat \bE_{\rho})\cdot \bar{\bH}_0 ds  = \int_{B_1}|\bE_0|^2 > 0 .\end{equation}
We claim that  
\begin{equation}\label{pro-O-2-e1}
\liminf_{\rho \to 0}\left|\int_{\partial B_1}(\nu\times \hat \bE_{\rho})\cdot \bar{\bH}_0 ds\right| =  0.
\end{equation}
Assuming this, we have, from \eqref{con:equa1-1},
$$
\liminf_{\rho \to 0} \left|  \int_{\partial B_1}(\nu\times \hat \bH_{\rho})\cdot \bar{\bE}_0  ds \right| > 0. 
$$
This implies, since $j_1(\omega)\neq 0$ by Lemma~\ref{opt-char}, that
$$
\liminf_{\rho \to 0} \left|  \int_{\partial B_1} \hat \bH_{\rho} \bar U^1_1  ds \right| > 0. 
$$
On the other hand,  by the separation of variables (see, e.g., \cite[Theorem 2.48]{Kirsch}), 
\begin{equation}\label{huhu}
   	\int_{\partial B_1} \hat \bH_\rho (r \hat x) \cdot  \bar U^{1}_1 (\hat x) \, d \hat x =  \frac{h^{(1)}_1(\omega\rho r) + \omega\rho rh'^{1}_{1}(\omega\rho r)}{r\big(h^{(1)}_1(\omega\rho ) + \omega\rho h'^{1}_{1}(\omega\rho)\big) }   \int_{\partial B_1} \hat \bH_{\rho}(\hat x) \cdot \bar U^1_1(\hat x)  d \hat x. 
   	\end{equation}
Using the fact
$$
\liminf_{\rho \to 0} \rho^{-2} \frac{1}{|h^{(1)}_1(\omega\rho ) + \omega\rho h'^{1}_{1}(\omega\rho) |}   > 0, 
$$
and taking $r = R/\rho$ with  $R\in (2, 4)$ in \eqref{huhu}, we obtain
$$
\liminf_{\rho\to 0}\rho^{-3} \int_{2}^4  \left|\int_{\partial B_1} \hat \bH_\rho (R\hat x/{\rho} ) \cdot  \bar U^{1}_1 (\hat x) \, d \hat x \right| \, d R > 0. 
$$
This implies, since $H_c (R \hat x) =\mathcal H_{\rho} (R \hat x) =  \rho^{-1} \bH_\rho( R \hat x / \rho)$ for $R \in (2, 4)$ and $\hat x  \in \partial B_1$, 
\begin{equation*}
\liminf_{\rho \to 0 } \rho^{-2 } \left\| H_{c}\right\|_{L^2(B_4\setminus B_2)} > 0, 
\end{equation*}
which is the conclusion.

\medskip 
It remains to prove \eqref{pro-O-2-e1}. Since
\begin{equation}\label{coucoucou}
	\bH_0 (x) = \frac{1}{i \omega} \nabla\times \bE_0 (x) =  \frac{2}{i\omega r}j_1(\omega r)Y_1^1 (\hat x) \hat x + \frac{1}{i\omega r}[j_1(\omega r) + \omega r j'_{1}(\omega r)] U_1^1 (\hat x ) \quad \mbox{ in } B_1, 
\end{equation}
	where $r = |x|$ and $\hat x = x / |x|$, using the  separation of variables (see, e.g., \cite[Theorem 2.48]{Kirsch}),  we have 
\begin{align}\label{pro-O-2-e2}
\liminf_{\rho \to 0}\left|\int_{\partial B_1}(\nu\times \hat \bE_{\rho})\cdot \bar{\bH}_0 \, d \hat x \right| &\leq C\liminf_{\rho \to 0}\left|\int_{\partial B_1}\hat \bE_{\rho}(\hat x)\cdot \bar V_1^1(\hat x) \,d\hat x\right| \\
\nonumber &= C\liminf\limits_{\rho\to 0}\left|\frac{-i\omega\rho}{\sqrt{2}}\int_{\partial B_1}\hat\bH_{\rho}(\hat x)|_{\ext} \cdot (\bar Y_1^1(\hat x)\hat x)\, d\hat x\right|.
\end{align}
Since, by Lemma \ref{lem:jump1}, 
\[\|\hat \bH_{\rho}\|_{H(\curl, B_5)} \leq C,\] 
we have 
\begin{equation}\label{pro-O-2-e3}
 \liminf\limits_{\rho\to 0}\left|\frac{-i\omega\rho}{\sqrt{2}}\int_{\partial B_1}\hat\bH_{\rho}|_{\ext} (\bar Y_1^1(\hat x)\hat x)\, d\hat x\right| = 0.
\end{equation}
Thus, \eqref{pro-O-2-e1} follows from \eqref{pro-O-2-e2} and \eqref{pro-O-2-e3}. \end{proof}

   We finally show the optimality of \eqref{thm1.2-est1} in the case where $J_{\ext} \equiv 0$ and $J_{\inte}$ does not satisfy the compatibility condition. 
   
   \begin{proposition}\label{pro-O-3}  Assume that $J_{\ext} = 0$ in $\R^3\setminus B_2$ and $j_1(\omega) = 0$. There exists $J_{\inte} \in [L^2(B_1)]^3$ such that 
   	\[\left\|E_{c}\right\|_{L^2(B_4\setminus B_2)} \geq C\rho,\]
   	for some positive constant $C$ independent of $\rho$.
   \end{proposition}
   
   \begin{proof}  Define  $J_{\inte}$ by \eqref{def-Jint}. We use the notations in the proof of Proposition~\eqref{pro-O-2}. We have
   	\begin{equation}\label{cccc} \int_{\partial B_1}(\nu\times \hat \bH_{\rho})\cdot \bar{\bE}_0  ds - \int_{\partial B_1}(\nu\times \hat \bE_{\rho})\cdot \bar{\bH}_0 ds  = \int_{B_1}|\bE_0|^2 > 0 .\end{equation}
   	Since $j_1(\omega) = 0$, it follows that 
	$$
	 \int_{\partial B_1}(\nu\times \hat \bH_{\rho})\cdot \bar{\bE}_0  ds = 0. 
	$$
	We derive from \eqref{cccc} that \footnote{This is the difference between the resonant and the non-resonant cases.}
	\[
\liminf_{\rho \to 0}\left|\int_{\partial B_1}(\nu\times \hat \bE_{\rho})\cdot \bar{\bH}_0 ds\right| > 0.
\]
This implies, by \eqref{coucoucou}, 
	\begin{equation}\label{order1-p3}\liminf\limits_{\rho \to 0}  \left| \int_{\partial B_1} \hat \bE_\rho (\hat x) \cdot  \bar V^{1}_1(\hat x) \, d \hat x \right| > 0.\end{equation}
By the separation of variables (see, e.g., \cite[Theorem 2.48]{Kirsch}), for $r > 2$, we obtain
	\begin{equation}\label{order1-p4}
	\int_{\partial B_1} \hat \bE_\rho (r \hat x) \cdot  \bar V^{1}_1 (\hat x) \, d \hat x =  \frac{h^{(1)}_1(\omega\rho r)}{h^{(1)}_1(\omega\rho )}  \int_{\partial B_1} \hat \bE_\rho (\hat x) \cdot \bar  V^{1}_1(\hat x) \, d \hat x. 
	\end{equation}
	Taking  $r = R/\rho$ with  $R\in (2, 4)$ in \eqref{order1-p4}, since  $\dsp
	\lim_{\rho \to 0} \rho^{-2} \left|\frac{h^{(1)}_1(\omega R)}{h^{(1)}_1(\omega\rho )} \right| > 0
	$,  we obtain   from \eqref{order1-p3}  that 
	$$
	\liminf_{\rho\to 0}\rho^{-2} \int_{2}^4  \left|\int_{\partial B_1} \hat \bE_\rho (R\hat x/{\rho} ) \cdot  \bar V^{1}_1 (\hat x) \, d \hat x \right| \, d R > 0. 
	$$
	This implies 
	\begin{equation*}
	\liminf_{\rho \to 0 } \rho^{-1} \left\| E_{c}\right\|_{L^2(B_4\setminus B_2)} > 0, 
	\end{equation*}
which is the conclusion. 
\end{proof}

\end{document}